\documentclass[11pt]{amsart}
\usepackage{verbatim}
\usepackage{amssymb}
\usepackage{colonequals}
\usepackage{mathtools}
\usepackage[bookmarks,colorlinks,breaklinks]{hyperref}  
\usepackage[top=2cm, bottom=4.5cm, left=3cm, right=3cm]{geometry}
\hypersetup{linkcolor=blue,citecolor=blue,filecolor=blue,urlcolor=blue} 
\usepackage{bm}
\usepackage{color}
\usepackage{amsfonts,amsmath,amsthm,amssymb,amscd, hyperref}

\numberwithin{equation}{section}

\newtheorem{theorem}{Theorem}[subsection]

\newtheorem{lemma}[theorem]{Lemma}
\newtheorem{proposition}[theorem]{Proposition}

\theoremstyle{definition}
\newtheorem{definition}[theorem]{Definition}
\newtheorem{definition-theorem}[theorem]{Definition-Theorem}
\newtheorem{example}[theorem]{Example}

\newtheorem{remark}[theorem]{Remark}

\theoremstyle{remark}

\newtheorem*{remark*}{Remark}

\usepackage{enumitem}
\setenumerate[1]{label=\textup{(\alph*)}}
\setenumerate[2]{label=\textup{(\roman*)}}

\newcommand{\Aberk}[1]{\ensuremath{\bA_{\mathrm{Berk},#1}^1}}
\newcommand{\g}{\bfc}

\newcommand{\E}{{\mathbb E}}
\renewcommand{\d}{\delta}
\renewcommand{\l}{\lambda}
\newcommand{\gal}{{\rm Gal}}

\newcommand{\Z}{{\mathbb Z}}
\newcommand{\bfc}{{\mathbf c}}
\newcommand{\N}{{\mathbb N}}
\newcommand{\M}{{\mathbb M}}
\newcommand{\Q}{{\mathbb Q}}
\newcommand{\C}{{\mathbb C}}
\newcommand{\A}{{\mathbb A}}
\newcommand{\bA}{{\mathbb A}}

\newcommand{\Qbar}{{\overline{\mathbb Q}}}

\newcommand{\Lbar}{{\overline{L}}}

\newcommand{\hhat}{\widehat{h}} 
\newcommand{\lra}{\longrightarrow}

\begin{document}
\raggedbottom
\title[Collision of orbits]{Collision of orbits for families of polynomials defined over number fields}

\author{Dragos Ghioca}
\address{Department of Mathematics, University of British Columbia, Vancouver, BC V6T 1Z2}
\email{dghioca@math.ubc.ca}

\author{Negin Shadgar}
\address{Department of Mathematics, University of British Columbia, Vancouver, BC V6T 1Z2}
\email{negin@math.ubc.ca}



\subjclass[2020]{Primary 37P15; Secondary 11G35}

\keywords{families of polynomials; unlikely intersections; collision of orbits}

\begin{abstract}
Let $d\ge 2$ be an integer and let $c_0(t),\dots, c_{d-2}(t)\in\Qbar[t]$. We consider the family of normalized polynomials $f_\l(z):=z^d+\sum_{i=0}^{d-2} c_i(\l)\cdot z^i$ parameterized by $\l\in\Qbar$; the generic element of our family of polynomials is $f_t(z):=z^d+\sum_{i=0}^{d-2}c_i(t)\cdot z^i\in \Qbar[t][z]$. Also, let $\alpha_1(t),\alpha_2(t),\beta(t)\in\Qbar[t]$, where $\alpha_i(t)$ is not preperiodic under the action of $f_t(z)$ for each $i=1,2$. Under some natural hypotheses, we obtain precise necessary and sufficient conditions for which there exist infinitely many $\l\in\Qbar$ with the property that for some $m,n\in\N$ (depending on $\l$), we have that $f_\l^m(\alpha_1(\l))=f_\l^n(\alpha_2(\l))=\beta(\l)$.
\end{abstract}

\maketitle

\section{Introduction}
\label{sec:intro}


\subsection{Notation}
For any given variety $X$ endowed with a self-map $\varphi$, the (forward) orbit of a point $x\in X$ under $\varphi$ is the set of all points $\varphi^n(x)\in X$ for $n\in\N$ (where $\varphi^n$ is the $n$-th compositional iterate of $\varphi$). We say that $x$ is \emph{preperiodic} under the action of $\varphi$ if its orbit is finite, i.e., there exist integers $0\le m<n$ such that $\varphi^m(x)=\varphi^n(x)$ (with the convention that $\varphi^0$ is the identity map). If $m=0$, i.e., $x=\varphi^n(x)$, then we say that $x$ is \emph{periodic} under the action of $\varphi$. On the other hand, if $x$ is preperiodic, but not periodic, then we say that $x$ is \emph{strictly preperiodic}.

Let $g(z)\in\Qbar[z]$ be a polynomial of degree $d\ge 2$. We say that $g$ is in \emph{normal form} (or equivalently, that it is \emph{normalized}) if it is monic and the coefficient of its $z^{d-1}$ term equals $0$. It is immediate to see that for any $g\in\Qbar[z]$, there exists a linear polynomial $\mu(z)\in\Qbar[z]$ such that $\mu^{-1}\circ g\circ \mu$ is in normal form; therefore, in our paper, we will focus on normalized polynomials.

We are interested in families of normalized polynomials parameterized as follows: given $c_0(t),\dots, c_{d-2}(t)\in\Qbar[t]$, then we have the family of polynomials $\left\{f_\l(z)\right\}_{\l\in\Qbar}$ given by 
\begin{equation}
\label{eq:-2}
f_\l(z)=z^d+\sum_{i=0}^{d-2}c_i(\l)\cdot z^i\text{, for each $\l\in\Qbar$.}
\end{equation}
The \emph{generic} element of our polynomial family is
\begin{equation}
\label{eq:-1}
f_t(z)=z^d+\sum_{i=0}^{d-2} c_i(t)\cdot z^i\in\Qbar[t][z].
\end{equation}
We identify the family of polynomials with its generic element $f_t(z)$, understanding that each element of our family of polynomials is obtained by specializing $t$ to some element $\l\in\Qbar$. When each polynomial  $c_i(t)$ (for $i=0,\dots, d-2$) is constant, then we say that $f_t(z)$ is a \emph{constant family of polynomials} (as it is represented simply by a polynomial in $\Qbar[z]$). For $\alpha(t)\in\Qbar[t]$, we say that $\alpha$ is \emph{persistent preperiodic} under the action of the  polynomial family $\{f_\l(z)\}_{\l\in\Qbar}$, if $\alpha(t)$ is  preperiodic under the action of $f_t(z)$. If $\alpha(t)$ is periodic under the action of $f_t(z)$,  then we say that $\alpha$ is \emph{persistent periodic} for our polynomial family. 

Finally, we recall the definition of the Chebyshev polynomials. For any integer $d\ge 2$, the \emph{$d$-th Chebyshev polynomial} $C_d(z)$ is the unique polynomial satisfying the identity:
\begin{equation}
\label{eq:Che}
C_d\left(x+\frac{1}{x}\right)=x^d+\frac{1}{x^d}.
\end{equation}


\subsection{Unlikely intersections in arithmetic dynamics}
\label{subsec:intro-unlikely}

Given the action of a self-map $\varphi$ on an ambient space $X$ (usually, $X$ is an algebraic variety and $\varphi$ is a regular morphism), one considers some \emph{unlikely event} $\mathcal{E}$  and asks which conditions allow the event $\mathcal{E}$ to occur infinitely often. The principle of unlikely intersections in arithmetic dynamics predicts that even though the event $\mathcal{E}$ is \emph{local} in nature, the fact that it is unlikely \emph{and} it repeats infinitely often  forces some \emph{strict global conditions} to be met by our dynamical system $(X,\varphi)$.  We list below a couple of examples involving a dynamical system $(X,\varphi)$ where $X$ is an algebraic variety defined over $\Qbar$ and $\varphi$ is a regular self-map, each of these examples  leading to deep questions in arithmetic dynamics:
\begin{itemize}
\item given an irreducible curve $C\subset X$, the event $\mathcal{E}$ occurs each time when a preperiodic point of $\varphi$ lands on $C$. Under some natural hypotheses, the expectation is that if this event occurs infinitely often (i.e., $C$ contains infinitely many preperiodic points), then $C$ must be preperiodic itself under the action of $\varphi$. This is the dynamical Manin-Mumford conjecture -- see \cite{Zhang, Sil-K3} for the origins of this question, and also, see \cite{GTZ, GT-BAMS} for more recent reformulations of the problem, which take into account certain special cases.
\item given an irreducible curve $C\subset X$ and also given a (non-preperiodic) point $x\in X$, the event $\mathcal{E}$ occurs each time when a point from the orbit of $x$ under $\varphi$ lies on $C$. Assuming this unlikely event occurs infinitely often, then the expectation is for the curve $C$ to be periodic under the action of $\varphi$ -- this is the dynamical Mordell-Lang conjecture (for more details, see \cite{BGT-book, G-DML}).   
\end{itemize}
For a comprehensive discussion of the unlikely intersection principle in arithmetic geometry, which is the starting point also for its dynamical counterpart, we refer the reader to the excellent book by Zannier \cite{Umberto}. 

The setting of algebraic dynamics is well-suited for versions of the unlikely intersections principle in \emph{families of maps}. For example, one could consider a family $\left\{f_\l(z)\right\}_{\l\in\Qbar}$ of polynomials as in \eqref{eq:-2} (with a generic element $f_t(z)\in\Qbar[t][z]$, as in \eqref{eq:-1}). Then given two \emph{starting points} generically being given by $\alpha_1(t),\alpha_2(t)\in\Qbar[t]$ (which are not persistent preperiodic), the unlikely event $\mathcal{E}$ occurs whenever for some $\l\in\Qbar$, we have that \emph{both} $\alpha_1(\l)$ and $\alpha_2(\l)$ are preperiodic points under the action of the polynomial $f_\l(z)\in\Qbar[z]$. If this event $\mathcal{E}$ occurs infinitely often, then the expectation is that $\alpha_1$ and $\alpha_2$ are \emph{(globally) dynamically related} with respect to our family of polynomials $f_t(z)$. More precisely, if the only  polynomials commuting with some iterate of $f_t$ are themselves iterates of $f_t$, then the expectation is that 
\begin{equation}
\label{eq:expectation}
\text{there exist positive integers $\ell_1$ and $\ell_2$ such that $f_t^{\ell_1}(\alpha_1(t))=f_t^{\ell_2}(\alpha_2(t))$.} 
\end{equation}
This problem was first formulated by Baker-DeMarco \cite{Matt}, but it has its origin in the work of Masser-Zannier \cite{M-Z-1, M-Z-2} regarding simultaneously torsion points in algebraic families of elliptic curves. The expected conclusion~\eqref{eq:expectation} (along with its extension for more general families of polynomials) was established in a series of papers  (see \cite{Matt-2, FG, GHT-ANT}). 
For more details regarding the unlikely intersection questions in arithmetic dynamics, we refer the reader to a couple of surveys \cite{survey, G-Survey}.

It is also natural to ask the following unlikely intersections question. Consider a family of polynomials given by $f_t(z)\in\Qbar[t][z]$ (see equations \eqref{eq:-1} and \eqref{eq:-2}), along with two \emph{starting} points $\alpha_1(t),\alpha_2(t)\in\Qbar[t]$ and one \emph{target} point $\beta(t)\in\Qbar[t]$. The unlikely event $\mathcal{E}$ occurs whenever for some $\l\in\Qbar$, we have that $\beta(\l)$ lies in \emph{both} the orbit of $\alpha_1(\l)$ and of $\alpha_2(\l)$ under the action of $f_\l(z)$. Once again, the expectation is that the infinite occurrence of the event $\mathcal{E}$ (which is coined as the \emph{collision of orbits} for the starting points $\alpha_1,\alpha_2$ at the target point $\beta$) should force a global dynamical relation between the three points $\alpha_1,\alpha_2,\beta$; this is what we obtain in our Theorem~\ref{thm:main}. 


\subsection{Our main result}

We prove the following statement.
\begin{theorem}
\label{thm:main}
Let $d\ge 2$, let $f_t(z)=z^d+\sum_{i=0}^{d-2} c_i(t)z^i\in \Qbar[t][z]$ be a family of polynomials, where $c_0(t),\dots, c_{d-2}(t)\in\Qbar[t]$. Also, let $\alpha_1(t),\alpha_2(t),\beta(t)\in\Qbar[t]$. Assume the following conditions hold:
\begin{enumerate}
\item[(I)] $c_0(t)$ is not identically equal to $0$;  
\item[(II)] not all polynomials $c_0(t),\dots, c_{d-2}(t)$, $\alpha_1(t),\alpha_2(t),\beta(t)$ are constant;
\item[(III)] there exists no $\gamma\in\Qbar^\ast$ such that $\gamma^{-1}f_t\left(\gamma z\right)=\pm C_d(z)$; and
\item[(IV)] for each $i=1,2$, $\alpha_i$ is not persistent preperiodic under the action of $f_t(z)$. 
\end{enumerate}
Then the set 
\begin{equation}
\label{eq:-0}
C_f(\alpha_1,\alpha_2;\beta):=\left\{\l\in\Qbar\colon \exists m,n\in\N\text{ such that }f_\l^m(\alpha_1(\l))=f_\l^n(\alpha_2(\l)) = \beta(\l)\right\}
\end{equation}
is infinite 
if and only if at least one of the following conditions hold:
\begin{itemize}
\item[(A)] there exists $\ell\in\N$ and there exists $i\in\{1,2\}$ such that $f_t^\ell\left(\alpha_i(t)\right)=\beta(t)$.

\item[(B)] there exist positive integers $\ell_1,\ell_2$  such that
$f_t^{\ell_1}\left(\alpha_1(t)\right) =f_t^{\ell_2}\left(\alpha_2(t)\right)$.

\item[(C)] there exists $\tilde{f}_t\in\Qbar[t][z]$ and there exist integers $k\ge 2$ and $\ell_1,\ell_2,\ell_3\ge 1$   such that 
$f_t=\tilde{f}_t^k$ and 
$\tilde{f}_t^{\ell_1}(\alpha_1(t))= \tilde{f}_t^{\ell_2}(\alpha_2(t))= \tilde{f}_t^{\ell_3}(\beta(t))$.

\item[(D)]  there exists $\tilde{f}_t\in\Qbar[t][z]$ and there exist integers $k\ge 2$ and $\ell_1,\ell_2\ge 1$   such that 
$f_t=\tilde{f}_t^k$ and 
$\tilde{f}_t^{\ell_1}(\alpha_1(t))=\tilde{f}_t^{\ell_2}(\alpha_2(t))$. 
Furthermore, $\beta$ is persistent periodic of minimal period $N_0$ under the action of $\tilde{f}_t$, where $N_0$ is a positive integer such that 
$\gcd (N_0, k) \mid  \ell_2 - \ell_1.$
\end{itemize}
\end{theorem}

Theorem~\ref{thm:main} answers a question raised in \cite[Remark~2.8]{A-G}. In \cite[Theorem~1.1]{A-G}, the collision of orbits problem was studied over fields of positive characteristic for the specific family of polynomials $f_t(z)=z^d+t$; due to working over a field of characteristic $p$, the problem from \cite{A-G} is already challenging when dealing even with the simplest of all polynomial families. In particular, the case of arbitrary families of polynomials over fields of positive characteristic is expected to be very difficult  (see \cite[Conjecture~2.6]{A-G}). In Theorem~\ref{thm:main} we answer the counterpart of the most general question posed in \cite[Section~2]{A-G} in   the case of polynomials over number fields.

Next, we make a couple of remarks regarding the hypotheses from Theorem~\ref{thm:main}; a much more in-depth discussion of the aforementioned hypotheses will be made in Section~\ref{sec:examples}.
\begin{remark}
\label{rem:not_constant}
We first note that hypothesis~(IV) is necessary for any meaningful statement since otherwise, assuming (say) that $\alpha_1$ is persistent preperiodic for $f_t(z)$, we would  either have that  $f_t^m(\alpha_1(t))=\beta(t)$ for some $m\in\N$, or there are only finitely many $\l\in\Qbar$ with the property that $f_\l^m(\alpha_1(\l))=\beta(\l)$ for some $m\in\N$.
\end{remark}

\begin{remark}
\label{rem:conclusions}
Second, we note that each conclusion~(A)-(D) from Theorem~\ref{thm:main} shows that there is \emph{at least} one dynamical relation between the two starting points $\alpha_1,\alpha_2$ and the target $\beta$ with respect to the family of polynomials $f_t(z)$. In particular, conditions~(C)-(D) appear due to the existence of other polynomials, not from the given family $f_t(z)$, which commute with  $f_t(z)$. As proven in Proposition~\ref{prop:poly}, the only polynomials commuting with an iterate of $f_t(z)$ (under the hypotheses of Theorem~\ref{thm:main}) are all compositional powers of some other polynomial $\tilde{f}_t$. 
\end{remark}

We also note that it makes sense to exclude from Theorem~\ref{thm:main} the case of constant families of polynomials $f_t(z)$, which are conjugated either to monomials, or to Chebyshev polynomials -- hence the relevance of hypotheses (I) and (III) from Theorem~\ref{thm:main}. Indeed, there is a clear dichotomy between polynomials which are dynamically related to endomorphisms of $\mathbb{G}_m$ (as it is the case of monomials and of the Chebyshev polynomials), as opposed to \emph{disintegrated polynomials} (see Definition~\ref{def:disintegrated}); this was studied thoroughly in \cite{Alice}. In Section~\ref{sec:examples}, we provide additional examples showing the novel subtleties appearing in the  collision of orbits problem in the context of non-disintegrated polynomials.


\subsection{Further connections between our Theorem~\ref{thm:main} and other studied questions}
\label{subsec:GCD}

The collision of orbits problem was already studied in several different settings: elliptic curves \cite{GHT-PJM, D-N}, Drinfeld modules \cite{G-JNT} and families of polynomials over fields of positive  characteristic \cite{A-G, LN}. At its core, the collision of orbits problem is a question about the greatest common divisor for families of polynomials; next, we explain this connection for the setting from Theorem~\ref{thm:main}.

So, with the notation as in Theorem~\ref{thm:main}, each $\l\in C_f(\alpha_1,\alpha_2;\beta)$ (see \eqref{eq:-0}) corresponds to an equation of the form:
\begin{equation}
\label{eq:220}
f_\l^m(\alpha_1(\l))=f_\l^n(\alpha_2(\l))=\beta(\l).
\end{equation}
Letting (for each $i=1,2$ and for each $k\in\N$) $P_{k,\alpha_i}(t):=f_t^k(\alpha_i(t))$, then $P_{k,\alpha_i}(t)\in\Qbar[t]$ (see also Section~\ref{subsec:notation_2}). Furthermore, letting $Q_{k,\alpha_i,\beta}(t):=P_{k,\alpha_i}(t)-\beta(t)$, then equation \eqref{eq:220} can be restated as saying that $\l$ is a root of $\gcd\left(Q_{m,\alpha_1,\beta}(t), Q_{n,\alpha_2,\beta}(t)\right)$. Thus, Theorem~\ref{thm:main} studies the greatest common divisor for two sequences of polynomials defined recursively (see Section~\ref{subsec:notation_2}); more precisely, it predicts that in the absence of conditions~(A)-(D) from Theorem~\ref{thm:main}, there are only finitely many distinct roots for all the possible \emph{GCD}'s of two polynomials, one from each sequence. 

Similar questions regarding \emph{GCD}'s were considered before, starting with the study of how large is $\gcd\left(a^m-1,b^n-1\right)$ (as we vary $m,n\in\N$) for two given multiplicatively independent positive integers $a$ and $b$ (see \cite{A-R, BCZ}). Motivated by the problems studied in \cite{A-R, BCZ}, similar \emph{GCD-problems} were considered in various settings: elliptic curves (see \cite{Silverman-3, GHT-PJM, Lau-1, D-N}), $S$-units (see \cite{Luca}), semiabelian varieties (see \cite{Lau-4}), algebraic dynamics (see \cite{HT}), including the study of the \emph{GCD-problem} over fields of characteristic $p$ (see \cite{C-Z-2, GHT-NJM}). As shown in Examples~\ref{ex:2}~and~\ref{ex:3}, the \emph{GCD-problem} raised by our dynamical question from Theorem~\ref{thm:main} may lead to difficult classical questions in the spirit of Artin's conjecture regarding primitive roots modulo primes.


\subsection{Strategy of proof and outline for our paper}
\label{subsec:strategy}

We start by presenting several examples in Section~\ref{sec:examples}, which show the relevance of each hypothesis from Theorem~\ref{thm:main}. We continue by setting up the notation for our proofs in Section~\ref{sec:technical}. In particular, in Section~\ref{subsec:poly}, we state several useful facts regarding commuting polynomials. 

In Section~\ref{sec:converse}, we prove the converse implication in Theorem~\ref{thm:main} (see Proposition~\ref{prop:converse}). More precisely, assuming each time that one of the conditions~(A)-(D) from Theorem~\ref{thm:main} holds, then we prove the corresponding set $C_f(\alpha_1,\alpha_2;\beta)$ (see \eqref{eq:-0}) is infinite. The main ingredient in proving Proposition~\ref{prop:converse} is Proposition~\ref{prop:real_converse}, which shows that given an active point $\alpha(t)$ for a family of polynomials $f_t(z)$ (see Definition~\ref{def:active}) and given a point $\beta(t)$, then there exist infinitely many $\l\in\Qbar$ such that $\beta(\l)$ is in the forward orbit of $\alpha(\l)$ under the action of $f_\l(z)$. Our Proposition~\ref{prop:real_converse} is in the spirit of \cite[Theorem~1.4]{Laura}, which establishes the existence of infinitely many $\l\in\Qbar$ such that $\alpha(\l)$ is preperiodic for $f_\l(z)$;  however, our proof is significantly different. As opposed to the arguments from \cite{Laura}, which rely mainly on complex dynamics techniques, the main ingredient in our proof of Proposition~\ref{prop:real_converse} is Diophantine, based on the classical Mordell-Lang conjecture for tori (which was settled by Laurent in \cite{Laurent}).

Sections~\ref{sec:heights}, \ref{subsec:Berkovich} and \ref{sec:proof_direct} are devoted to proving the direct implication in Theorem~\ref{thm:main}, which is stated as Proposition~\ref{prop:direct}. More precisely, we assume the set $C_f(\alpha_1,\alpha_2;\beta)$ from Theorem~\ref{thm:main} is infinite and we prove in Proposition~\ref{prop:direct} that at least one of the conclusions~(A)-(D) must hold. In Section~\ref{sec:heights}, we introduce the canonical heights $\hhat_\l:=\hhat_{f_\l}$ associated to the family of polynomials  from Theorem~\ref{thm:main} (see \eqref{eq:global_canonical}). We derive various useful facts, including Proposition~\ref{prop:heights}, which shows that under some natural hypotheses in Theorem~\ref{thm:main}, the fact that the set $C_f(\alpha_1,\alpha_2;\beta)$ is infinite yields the existence of an infinite sequence of parameters $\l_k\in\Qbar$ with the property that
\begin{equation}
\label{eq:221}
\lim_{k\to\infty} \hhat_{\l_k}(\alpha_1(\l_k))= \lim_{k\to\infty} \hhat_{\l_k}(\alpha_2(\l_k))=0.
\end{equation}
Equation \eqref{eq:221} allows us to apply the equidistribution theorem of Baker-Rumely \cite[Theorem~7.52]{BR} and thus get that for \emph{each} $\l\in\Qbar$, we have that $\hhat_\l(\alpha_1(\l))=0$ if and only if $\hhat_\l(\alpha_2(\l))=0$; for more details, see Section~\ref{subsec:Berkovich}. In particular, this means (see also \cite[Theorem~1.4]{Laura}) that there exist infinitely many $\l\in\Qbar$ such that \emph{both} $\alpha_1(\l)$ and $\alpha_2(\l)$ are preperiodic points for $f_\l(z)$ (see Theorem~\ref{thm:iff}). Then \cite[Theorem~1.3]{Matt-2} yields that $\alpha_1(t)$ and $\alpha_2(t)$ are (globally) dynamically related with respect to $f_t(z)$; in particular, this allows us to derive conclusions~(A)-(B)  in Proposition~\ref{prop:direct}, assuming the only polynomials commuting with some iterate of $f_t(z)$ are themselves polynomial iterates of $f_t(z)$. Using the refinements from Section~\ref{subsec:poly} regarding commuting polynomials, we can derive the precise conditions~(C)-(D) in case $f_t$ admits more commuting polynomials (than its compositional iterates) and the dynamical relations between the points $\alpha_1,\alpha_2,\beta$ involve these additional polynomials.


\section{More examples and further comments regarding  Theorem~\ref{thm:main}}
\label{sec:examples}

We recall the definition of $C_f(\alpha_1,\alpha_2;\beta)$, which is the set of all $\l\in\Qbar$ with the property that there exist some $m,n\in\N$ such that
\begin{equation}
\label{eq:0}
f_\l^m(\alpha_1(\l))=f_\l^n(\alpha_2(\l))=\beta(\l).
\end{equation}
We show various examples in which the absence of hypotheses~(I)-(III) gives rise to complications in Theorem~\ref{thm:main}; we note that the relevance of hypothesis~(IV) was already explained in Remark~\ref{rem:not_constant}.


\subsection{The case of a constant family of polynomials}
We start by presenting an example in which both the family $f_t(z)$ and also the starting points $\alpha_1(t),\alpha_2(t)$, along with the target point $\beta(t)$ are all constant. In this case, condition~(A) no longer guarantees that the set $C_f(\alpha_1,\alpha_2;\beta)$ would be  infinite. 

\begin{example}
\label{ex:1}
Assume each polynomial $c_i(t)$ is constant, for $i=0,\dots, d-2$. Furthermore, assume $\alpha_1(t),\alpha_2(t),\beta(t)$ are also constant. Then condition~(A) (but also, conditions~(C)-(D)) do not guarantee the existence of infinitely many $\l\in\Qbar$ as in equation~\eqref{eq:0}. Indeed, the fact that $\beta$ is in the orbit of $\alpha_1$ (which is condition~(A)) does not have implications regarding $\alpha_2(\l)$ (which always equals $\alpha_2$ in Example~\ref{ex:1}) having an orbit including $\beta(\l)$ (which also is constant in our Example~\ref{ex:1}). So, one would need to assume that at least one of the polynomials $c_0(t),\dots,c_{d-2}(t),\alpha_1(t),\alpha_2(t),\beta(t)$ is nonconstant, i.e., the hypothesis~(II) from Theorem~\ref{thm:main} is needed.
\end{example}

Next, we show that if our family of polynomials $f_t(z)$ does not satisfy hypothesis~(I), then Theorem~\ref{thm:main} may not hold.


\subsection{Polynomial families commuting with nontrivial linear polynomials}

First, Example \ref{ex:5} shows that if one allows for polynomial families commuting with nontrivial linear polynomials (which might happen in the absence of hypothesis~(I) from Theorem~\ref{thm:main}), one would have to alter conditions~(C)-(D) to allow for more complex dynamical relations between the points $\alpha_1,\alpha_2,\beta$.

\begin{example}
\label{ex:5}
Let $f_t(z):=z^3+tz$, $\alpha_1(t)=1$, $\alpha_2(t)=-1$ and $\beta(t)=1$. Then conditions~(A)-(D) from Theorem~\ref{thm:main} do not hold (note that in condition~(A), we ask that $\beta(t)$ is in the \emph{forward} orbit of $\alpha_i(t)$ and thus, the case $\alpha_1=\beta$ does not satisfy condition~(A)). However, in this case, we still have infinitely many $\l\in\Qbar$ (and $m,n\in\N$) satisfying equation~\eqref{eq:0}. Indeed, for each $m\in\N$, we solve the equation
\begin{equation}
\label{eq:201}
f_\l^m(-1)=1.
\end{equation}
As proven in Proposition~\ref{prop:real_converse}, there exist infinitely many $\l\in\Qbar$ satisfying \eqref{eq:201} as we vary $m\in\N$. But then, for each such $\l\in\Qbar$ and $m\in\N$ as in equation \eqref{eq:201}, we also get that
\begin{equation}
\label{eq:202}
f_\l^{2m}(1)=f_\l^m\left(f_\l^m(1)\right)=f_\l^m(-1)=1,
\end{equation} 
because $f_t(z)$ is a family of odd polynomials. Equations \eqref{eq:201} and \eqref{eq:202} (along with Proposition~\ref{prop:real_converse}) show that there are infinitely many $\l\in\Qbar$ (and $m,n\in\N$) satisfying equation \eqref{eq:0}, even though none of the conditions~(A)-(D) from Theorem~\ref{thm:main} hold.
\end{example}

The issue presented in this Example~\ref{ex:5} comes from the fact that our family of polynomials $f_t(z)$ admits more polynomials commuting with (some iterate of it) due to the extra symmetry presented by the map $z\mapsto -z$. This suggests that if we were to drop hypothesis~(I) from Theorem~\ref{thm:main}, then conditions~(C)-(D) should be suitably modified to allow other polynomials commuting with (some iterate of) $f_t(z)$ beyond the compositional iterates of the polynomials $\tilde{f}_t(z)$ for which $f_t=\tilde{f}_t^k$ (for some $k\ge 2$). Nevertheless, any such modification may still lead to complications as Example~\ref{ex:4}  shows.

\begin{example}
\label{ex:4}
Let $\tilde{f}_t(z) = z^3 + t\cdot z$ and let $f_t = \tilde{f}_t^2$.  Then let $\alpha_1(t) = -2$, $\alpha_2(t) = \tilde{f}_t(2)=2t+8$ and $\beta(t) = 2$. 

Since $\tilde{f}_t(z)$ is an odd polynomial, there are extra global dynamical relations coming from the map $z\mapsto -z$; in particular,  a suitably modified  conclusion~(C) from Theorem~\ref{thm:main} holds in this case for the three points $\alpha_1,\alpha_2,\beta$. However, we claim that there are \emph{no} $\lambda\in\Qbar$ such that equation \eqref{eq:0} holds for some  positive integers $m$ and $n$ (depending on $\l$). Indeed, if for some $\l\in\Qbar$ and $m,n\in\N$, we have 
\begin{equation}
\label{eq:128}
f_\lambda^m(\alpha_1(\lambda)) = f_\lambda^n(\alpha_2(\lambda)) =
\beta(\lambda), 
\end{equation}
then due to our choice for $\alpha_1, \alpha_2, \beta$ (along with $f_t = \tilde{f}_t^2$), we get that $2$ must be periodic under the action of $\tilde{f}_\lambda$.  Furthermore, since \eqref{eq:128} yields that $\tilde{f}_\lambda^{2n+1}(2) = 2$, we obtain that 
\begin{equation}
\label{eq:130}
\text{the minimal period of $2$ under the action of $\tilde{f}_\lambda$ is odd.} 
\end{equation}
On the other hand, \eqref{eq:128} yields  that
$\tilde{f}_\lambda^{2m}(-2) = 2$ 
and so, using that $\tilde{f}_\lambda$ is an odd polynomial, we have
$\tilde{f}_\lambda^{2m}(2) = -2$; hence,
\begin{equation}
\label{eq:129}
\tilde{f}_\lambda^{4m}(2) = \tilde{f}_\lambda^{2m}(-2) = 2.
\end{equation}
Equation \eqref{eq:129} yields that the period of $2$ under the action of $\tilde{f}_\lambda^{2m}$ is $2$, but on
the other hand, it must be odd (see \eqref{eq:130}); thus, we have
a contradiction, which means that equation~\eqref{eq:128} never holds. Hence, the absence of hypothesis~(I) may lead to a different conclusion in  Theorem~\ref{thm:main}.  
\end{example}

Next, we show that a very small change in Example~\ref{ex:4} might lead to a completely opposite conclusion.
\begin{example}
\label{ex:6}
We work again with  $\tilde{f}_t(z) = z^3 + t\cdot z$, $f_t = \tilde{f}_t^2$, $\alpha_1(t)= -2$ and $\beta(t)=2$. However, this time, we let $\alpha_2(t) =f_t(2)= \tilde{f}^2_t(2)$. Once again, due to the fact that $\tilde{f}_t(z)$ is an odd polynomial (i.e., it commutes with the linear polynomial $L(z)=-z$), we get that $\alpha_1(t)$ and $\alpha_2(t)$ (and also, $\beta(t)$) are dynamically related. So, we have the same general scenario as in Example~\ref{ex:4}, but as opposed to that case (in which we had no parameters $\l\in\Qbar$ leading to a collision of orbits), this time we claim there are infinitely many parameters $\l\in\Qbar$ such that for some suitable $m,n\in\N$, we would have
\begin{equation}
\label{eq:210}
f_\l^m(\alpha_1(\l))=f_\l^n(\alpha_2(\l))=\beta(\l).
\end{equation} 
Indeed, Proposition~\ref{prop:real_converse} yields that there exist infinitely many $\l\in\Qbar$ such that for some $m\in\N$, we have
\begin{equation}
\label{eq:211}
f_\l^m(-2)=2;
\end{equation}
then equation \eqref{eq:211} (along with the fact that $f_\l(z)$ is an odd polynomial) yields that $f_\l^{m}(2)=-2$ and thus, $f_\l^{2m}(2)=2$. Therefore, equation \eqref{eq:210} holds for infinitely many $\l\in\Qbar$, as claimed.
\end{example}

Examples~\ref{ex:5}, \ref{ex:4} and \ref{ex:6} show that allowing for the family $f_t(z)$ to have linear symmetries (i.e., more precisely, that there exist nontrivial linear polynomials commuting with some iterate of $f_t(z)$) leads to scenarios in which it becomes unclear what is the precise relation which guarantees the set $C_f(\alpha_1,\alpha_2;\beta)$ is infinite. In particular, this shows the relevance of hypothesis~(I) from Theorem~\ref{thm:main}; note that Proposition~\ref{prop:poly} yields that hypothesis~(I) prevents the existence of nontrivial linear symmetries for our family of polynomials. 


\subsection{The relevance for hypothesis~(III) in Theorem~\ref{thm:main}}

Examples~\ref{ex:2}~and~\ref{ex:3} show the connections of our problem to some (potentially deep) Diophantine questions; in particular, they justify hypothesis~(III) in Theorem~\ref{thm:main} and also, provide further justification to hypothesis~(I).

\begin{example}
\label{ex:2}
Next assume that the polynomials $c_0(t),\dots, c_{d-2}(t)$ are constant, while the polynomials $\alpha_1(t),\alpha_2(t),\beta(t)$ are not necessarily constant. Then the special case $f_t(z)=z^d$, i.e., 
\begin{equation}
\label{eq:-11}
c_0(t)=\cdots =c_{d-2}(t)=0
\end{equation}
leads to potentially  deep Diophantine questions. In this case, $f_t(z)=z^d$ commutes with more polynomials than the normal expectation for a polynomial of degree $d$ (see also Proposition~\ref{prop:commutation}); more precisely, $f_t(z)$ commutes in this case with any constant monomial family.  So, in light of Remark~\ref{rem:conclusions}, conclusions~(C)-(D) should be altered in order to allow the various global relations between the three points $\alpha_1,\alpha_2,\beta$ with respect to monomial actions.  Thus, consider the special case $\alpha_1(t)=t^2$, $\alpha_2(t)=t^3$ and $\beta(t)=t^6$. 
In this case, the three points $\alpha_1,\alpha_2,\beta$ are  dynamically related according to a modified version of conclusion~(C); indeed, letting $g_1(z)=z^{3}$, $g_2(z)=z^{2}$ and $g_3(z)=z$, we have:
\begin{equation}
\label{eq:120}
g_1(t^2)=g_2(t^3)=g_3(t).   
\end{equation}
Equation \eqref{eq:120} suggests that we \emph{should} be able to find infinitely many $\l\in\Qbar$ satisfying equation~\eqref{eq:0} for some suitable $m,n\in\N$. Equation~\eqref{eq:0} leads to solving for some $m,n\in\N$ the equation
\begin{equation}
\label{eq:122}
\l^{2d^m}=\l^{3d^n}=\l^6\text{, for $\l\in\Qbar$.}
\end{equation}
So, we would have infinitely many solutions $\l$ to equation~\eqref{eq:122} if and only if there exist infinitely many positive integers $D$ with the property that for some suitable $m,n\in\N$, we have $D\mid \gcd\left(d^m-2,d^n-3\right)$. This last question is difficult; more generally, given (coprime) positive integers $d,a,b$, the question 
\begin{equation}
\label{eq:121}
\text{whether the set $\left\{\gcd\left(d^m-a,d^n-b\right)\colon m,n\in\N\right\}$ is infinite}
\end{equation}
is most likely beyond the current Diophantine techniques (perhaps, being closely related to the famous Artin's conjecture regarding primitive roots modulo primes). In particular, we should exclude in Theorem~\ref{thm:main} the possibility that $f_t(z)=z^d$; this is done by assuming $c_0(t)\ne 0$, which is hypothesis~(I) from Theorem~\ref{thm:main}.
\end{example}

Furthermore, we can construct a twist of Example~\ref{ex:2} using Chebyshev polynomials (see equation \eqref{eq:Che} for the defining property of the Chebyshev polynomials). Note that equation \eqref{eq:Che} yields the identity
\begin{equation}
\label{eq:124}
C_N\circ C_M = C_{NM}\text{ for any integers $N,M\ge 2$.}
\end{equation}

\begin{example}
\label{ex:3}
Assume now that $f_t(z)\in\Qbar[z]$ is the constant polynomial family equal to the $d$-th Chebyshev polynomial $C_d(z)$. Furthermore, given integers $a,b>1$ coprime with $d$, assume $\alpha_1(t)=C_a(t)$, $\alpha_2(t)=C_b(t)$ and also, $\beta(t)=t$. Then due to commutation~\eqref{eq:124}, we see that a suitable modification of conclusion~(C) from Theorem~\ref{thm:main} would hold for our example,  which suggests that we should be able to find infinitely many $\l\in\Qbar$ satisfying equation~\eqref{eq:0} for some  positive integers $m$ and $n$ (depending on $\l$). But then for any such solutions $\l\in\Qbar$ and $m,n\in\N$, we would have that
\begin{equation}
\label{eq:125}
C_{d^ma}(\l)=C_{d^nb}(\l)=\l.
\end{equation} 
Writing $\l=u+\frac{1}{u}$ for some $u\in\Qbar^\ast$, equations \eqref{eq:125} and \eqref{eq:Che} lead then
\begin{equation}
\label{eq:126}
u^{\pm d^ma}=u^{\pm d^nb}=u.
\end{equation}
Thus, equation \eqref{eq:126} has infinitely many solutions if and only if the set 
\begin{equation}
\label{eq:127}
\left\{\gcd\left(ad^m\pm 1,bd^n\pm 1\right)\colon m,n\in\N\right\}\text{ is infinite.}
\end{equation}
As discussed in Example~\ref{ex:2}, the problem raised in \eqref{eq:127} is expected to be quite challenging. Therefore,  hypothesis~(III) is necessary in Theorem~\ref{thm:main} (also, note that $C_d(z)$ is a normalized polynomial and so, by Lemma~\ref{lem:conjugated}, any normalized polynomial conjugated to $\pm C_d(z)$ is of the form $\pm \gamma^{-1}C_d(\gamma\cdot z)$). 
\end{example}


\section{Technical setup and some useful facts}
\label{sec:technical}

In Section~\ref{subsec:notation_2}, we introduce the setup employed throughout our proofs for both the converse and the direct implication in Theorem~\ref{thm:main}. In particular, we define the canonical height for elements $\alpha(t)\in\Qbar[t]$ with respect to the generic element of a  polynomial family $f_t(z)\in\Qbar[t][z]$ (see \eqref{eq:canonical_generic}). 

In Section~\ref{subsec:poly}, we derive a useful result (see Proposition~\ref{prop:poly}), which will be employed in our proof of the direct implication of Theorem~\ref{thm:main}. Proposition~\ref{prop:poly} is based on a classical fact (see Proposition~\ref{prop:commutation}) regarding the classification of the set of polynomials commuting with an iterate of a given polynomial. We also state a few easy facts regarding conjugated polynomials in Section~\ref{subsec:conjugated}.


\subsection{Canonical height with respect to the generic element of a polynomial family}
\label{subsec:notation_2}

Let $f_t(z)\in\Qbar[t][z]$ be the generic element of a polynomial family of degree $d\ge 2$. For each $\alpha\in\Qbar[t]$ and for each $n\in\N$, we define $P_{f,n,\alpha}(t):=f_t^n\left(\alpha(t)\right)$; when the polynomial family is implicitly understood, we drop the dependence on $f$ from the index of $P_{f,n,\alpha}(t)$ and simply write
\begin{equation}
\label{eq:Pn}
P_{n,\alpha}(t)=f_t^n\left(\alpha(t)\right).
\end{equation}

\begin{lemma}
\label{lem:iterates}
Let $d\ge 2$, let $c_0(t),\dots,c_d(t)\in\Qbar[t]$ (with $c_d(t)\ne 0$)  and let
\begin{equation}
\label{eq:gen_fam}
f_t(z):=\sum_{i=0}^d c_i(t)\cdot z^i\in\Qbar[t][z].
\end{equation}
Let $\alpha\in\Qbar[t]$ and let $P_{n,\alpha}(t)$ be defined as in equation~\eqref{eq:Pn}. If 
\begin{equation}
\label{eq:alpha}
\deg_t(\alpha(t))> \max_{i=0}^{d-1}\deg_t(c_i(t)), 
\end{equation}
then for each $n\in\N$, we have
\begin{equation}
\label{eq:gen_form}
\deg_t\left(P_{n,\alpha}(t)\right)= \deg_t(\alpha(t))\cdot d^n + \deg_t(c_d(t))\cdot \frac{d^n-1}{d-1}.
\end{equation}
\end{lemma}

\begin{proof}
For the sake of simplifying our notation, for each $n\in\N$, we let $D_n:=\deg_t\left(P_{n,\alpha}(t)\right)$. Then the formula~\eqref{eq:gen_form} follows easily by induction on $n$, once we establish the recurrence relation
\begin{equation}
\label{eq:reccur}
D_{n+1}=d\cdot D_n + \deg_t(c_d(t))\text{ for each $n\in\N$.}
\end{equation}
The recurrence relation \eqref{eq:reccur} follows immediately once we note that whenever $\gamma\in\Qbar[t]$ satisfies the following  inequality  for its degree:
\begin{equation}
\label{eq:gamma}
\deg_t(\gamma(t))>\max_{i=0}^{d-1} \deg_t(c_i(t)),
\end{equation}
then we have
\begin{equation}
\label{eq:gamma_2}
\deg_t\left(f_t(\gamma(t))\right) = d\cdot \deg_t(\gamma(t)) + \deg_t(c_d(t))>\deg_t(\gamma(t)).
\end{equation}
Now, our hypothesis \eqref{eq:alpha} yields that $\alpha(t)$ and (due to using iteratively \eqref{eq:gamma_2}) also all its iterates $f_t^n(\alpha(t))$ satisfy the inequality~\eqref{eq:gamma} for their respective degrees $D_n$. Therefore, the sequence of degrees $D_n$ satisfies the recurrence formula \eqref{eq:reccur}, thus proving the desired formula~\eqref{eq:gen_form}. This concludes our proof of Lemma~\ref{lem:iterates}. 
\end{proof}

Given a polynomial family $f_t(z)\in\Qbar[t][z]$ of degree $d\ge 2$ and given $\alpha\in\Qbar[t]$, we define the (global) canonical height of $\alpha$ with respect to $f_t$ (and we denote it by  $\hhat_t(\alpha(t))$) as follows:
\begin{equation}
\label{eq:canonical_generic}
\hhat_t(\alpha(t)):=\lim_{n\to\infty} \frac{\deg_t\left(f_t^n(\alpha(t))\right)}{d^n}= \lim_{n\to\infty} \frac{\deg_t\left(P_{n,\alpha}(t)\right)}{d^n}
\end{equation}
(see also equation~\eqref{eq:Pn} regarding the notation for $P_{n,\alpha}(t)$). Since each coefficient of $f_t(z)$ is itself a polynomial in $t$ (and thus, integral at all places of $\Qbar(t)$ away from the place at infinity), our formula~\eqref{eq:canonical_generic} coincides with the classical definition of the canonical height for polynomials defined over function fields (see \cite[Proposition~1.2]{Call-Silverman} and also, \cite[Section~3.4]{Silverman}). 

Following \cite{Laura}, we define \emph{active points} for a polynomial family as follows.
\begin{definition}
\label{def:active}
Given $f_t(z)\in\Qbar[t][z]$ and $\alpha(t)\in\Qbar[t]$, we say that $\alpha$ is an active point for $f_t$ if $\hhat_t(\alpha(t))>0$.
\end{definition}
Clearly, if $\alpha$ is persistent preperiodic for $f_t$, then $\alpha$ is not active. 
The following result proven in \cite[Theorem~A]{Rob} (see also \cite[Corollary~1.8]{Matt-0} for a generalization of this result to rational functions) shows that the only non-active, non-preperiodic points arise in the \emph{isotrivial} setting.
\begin{proposition}
\label{prop:nonzero_height}
Let $f_t(z)\in\Qbar[t][z]$ and let $\alpha(t)\in\Qbar[t]$. If 
\begin{itemize}
\item $\alpha$ is not persistent preperiodic for $f_t$;  and 
\item $\hhat_t(\alpha(t))=0$,
\end{itemize}
then $f_t(z)$ is isotrivial, i.e.,  there exists a linear polynomial $\mu\in\overline{\Q(t)}[z]$ such that $\mu^{-1}\circ f_t\circ \mu\in\Qbar[z]$.
\end{proposition}

In the next Section~\ref{subsec:conjugated}, we state a few easy facts regarding conjugated polynomials (we recall that two polynomials $g_1(z)$ and $g_2(z)$ are \emph{conjugated} if there exists a linear polynomial $\mu(z)$ such that $g_2\circ \mu=\mu\circ g_1$).


\subsection{Conjugated polynomials}
\label{subsec:conjugated}

In this Section~\ref{subsec:conjugated}, we work with polynomials in $\C[z]$. Since we can always embed $\Qbar(t)$ into $\C$, the generic element $f_t(z)\in\Qbar[t][z]$ of a polynomial family can always be seen as an element of $\C[z]$.

The following result is standard (see \cite[Lemma~10.2]{GHT-ANT}).
\begin{lemma}
\label{lem:iso}
Let $f\in\C[z]$ be a polynomial in normal form. Then $f$ is conjugated to a polynomial in $\Qbar[z]$ if and only if $f\in\Qbar[z]$.
\end{lemma}

The following result will be employed in the proof of Proposition~\ref{prop:commutation}.

\begin{lemma}
\label{lem:conjugated}
Let $g_1,g_2\in\C[z]$ be polynomials of degree $d\ge 2$ and let $\mu(z)=Az+B\in\C[z]$ such that $g_1\circ \mu = \mu\circ g_2$. If the coefficients of $z^{d-1}$ in both $g_1(z)$ and $g_2(z)$ equal $0$, then $B=0$.  
\end{lemma}

\begin{proof}
The proof is immediate since the coefficient of $z^{d-1}$ in $\mu\circ g_2$ equals $0$, while the coefficient of $z^{d-1}$ in $g_1\circ \mu$ equals $0$ precisely when $B=0$.
\end{proof}

Following \cite{Alice}, we define the \emph{disintegrated polynomials} as follows.
\begin{definition}
\label{def:disintegrated}
A polynomial $g(z)\in\C[z]$ of degree $d\ge 2$ is called disintegrated, if it is not conjugated either to $z^d$, or to $\pm C_d(z)$. 
\end{definition}

\begin{remark}
\label{rem:disintegrated}
Let $g(z)\in\C[z]$ be a polynomial of degree $d\ge 2$ with the property that the coefficient of $z^{d-1}$ in $g(z)$ equals $0$.  
Lemma~\ref{lem:conjugated} yields that $g(z)$ is disintegrated if and only if there exists no $b\in\C^\ast$ such that $b^{-1}g(bz)$ equals either $z^d$, or $\pm C_d(z)$. 
\end{remark}


\subsection{Polynomials commuting with an iterate of a given polynomial}
\label{subsec:poly}
The following result will be used in establishing conclusions~(C)-(D) in the direct implication for Theorem~\ref{thm:main}.

\begin{proposition}
\label{prop:poly}
Let $f_t(z)\in\Qbar[t][z]$ be a polynomial family of degree $d\ge 2$  given by 
\begin{equation}
\label{eq:00}
f_t(z)=z^d+\sum_{i=0}^{d-2} c_i(t)\cdot z^i,
\end{equation} 
satisfying the following two hypotheses:
\begin{itemize}
\item[(i)] $c_0(t)\ne 0$; and
\item[(ii)] there exists no $b\in\Qbar^\ast$ such that $b^{-1}f_t(b\cdot z)=\pm C_d(z)$.
\end{itemize} 
Then there exists a polynomial $\tilde{f}_t(z)\in\Qbar[t][z]$ such that the set of polynomials commuting with some iterate of $f_t$ is precisely the set $\left\{\tilde{f}_t^k\colon k\in\N\cup\{0\}\right\}$.
\end{proposition}

The key to proving Proposition~\ref{prop:poly} lies in the following result appearing in \cite[Proposition~2.8]{ANT}.

\begin{proposition}
\label{prop:commutation}
Let $g\in\C[z]$ be a polynomial of degree $d\ge 2$. We let 
$$\mathcal{C}_g:=\left\{h\in\C[z]\colon \deg(h)\ge 1\text{ and there exists $\ell\in\N$ such that } h\circ g^\ell=g^\ell\circ h\right\}.$$
Then there exists $\tilde{g}, L\in\C[z]$ with $L$ being a linear polynomial satisfying the following properties:
\begin{itemize}
\item $L$ has finite order, i.e., there exists $r\in\N$ such that $L^r(z)=z$.
\item There exists $D\in\N$ coprime with $r$ such that $L^D\circ \tilde{g}=\tilde{g}\circ L$.
\item We have that
\begin{equation}
\label{eq:C_g}
\mathcal{C}_g=\left\{L^i\circ \tilde{g}^j\colon 0\le i\le r-1\text{ and }j\ge 0\right\}.
\end{equation} 
\end{itemize}
\end{proposition}

Now, we can prove the main result of Section~\ref{subsec:poly}.

\begin{proof}[Proof of Proposition~\ref{prop:poly}.]
First, we view $f_t$ as a polynomial in $\C[z]$ (by considering an embedding of $\Qbar(t)$ into $\C$).

Now, our goal is to show that there exists no nontrivial linear polynomial commuting with an iterate of $f_t$. So, let $L$ be the linear polynomial commuting with an iterate of $f_t$ appearing in the conclusion of Proposition~\ref{prop:commutation} (see also equation~\eqref{eq:C_g}). 

First, we claim that $L(0)=0$. Indeed, we know that $L$ commutes with some iterate $f_t^m$ (for some $m\in\N$). Because $f_t$ is in normal form, then also $f_t^m$ is a normalized polynomial, i.e., $f_t^m(z)$ has degree $d^m$, is monic and its coefficient of $z^{d^m-1}$ equals $0$. But then writing $L(z)=az+b$, imposing the condition that $L\circ f_t^m=f_t^m\circ L$ yields that $b=0$ (see Lemma~\ref{lem:conjugated}), as claimed. 

Now, let $\tilde{f}_t\in\C[z]$ be as in the conclusion of Proposition~\ref{prop:commutation} (see also equation~\eqref{eq:C_g}). Then we know that $L^D\circ \tilde{f}_t=\tilde{f}_t\circ L$ for some positive integer $D$. We already know that $L(z)=a z$ and so, equating the coefficient of $z^0$ in $L^D\circ \tilde{f}_t=\tilde{f}_t\circ L$, we get that $a^D=1$ (note that $\tilde{f}_t(0)\ne 0$ because otherwise, we would get that $f_t(0)=0$, contradiction). However, as noted in the conclusion of Proposition~\ref{prop:commutation}, $D$ must be coprime with the order of $L$, which contradicts that $a^D=1$, unless $a=1$. So, indeed, we conclude that the only linear polynomial commuting with some iterate of $f_t$ is the identity polynomial itself.

So, we obtained (see again Proposition~\ref{prop:commutation}) that all polynomials commuting with some iterate of $f_t$ are themselves iterates of some suitable polynomial $\tilde{f}_t$ (for which $\tilde{f}_t^k=f_t$ for some $k\in\N$). We claim that $\tilde{f}_t\in\Qbar[t][z]$, as desired in the conclusion of Proposition~\ref{prop:poly}.

First, we get that $\tilde{f}_t\in\Qbar(t)[z]$ since otherwise, we could pick some  $\C$-automorphism $\sigma$ which fixes $\Qbar(t)$, but which does not fix all coefficients of $\tilde{f}_t(z)$. Then $\tilde{f}_t^{\sigma}$ also commutes with an iterate of $f_t$ and it has the same degree as $\tilde{f}_t$; this contradicts the fact that all polynomials commuting with $f_t$ are iterates of $\tilde{f}_t$. So, this contradiction shows that we must have that $\tilde{f}_t\in\Qbar(t)[z]$.

Second, we get that the coefficients of $\tilde{f}_t(z)$ (which are themselves rational functions from $\Qbar(t)$) do not have poles, since otherwise any iterate of $\tilde{f}_t(z)$ would also have poles, thus contradicting that $\tilde{f}_t^k=f_t\in\Qbar[t][z]$. Indeed, writing

$$
\tilde{f}_t(z)=\sum_{j=0}^e b_j(t)\cdot z^j\text{ for some $b_j\in \Qbar(t)$},
$$
using that $\tilde{f}_t^k$ is a polynomial in normal form, we get that $b_e(t)$  must be a constant polynomial and furthermore, $b_{e-1}$ must be the zero polynomial. Then letting $j_0$ be the largest index for which we assume $b_{j_0}\notin\Qbar[t]$, we immediately get (using induction on $k$) that the coefficient of $z^{e^k-(e-j_0)}=z^{d-(e-j_0)}$ in $\tilde{f}_t^k(z)$ must also have a pole (as a rational function in $\Qbar(t)$), thus contradicting the fact that $\tilde{f}_t^k=f_t\in \Qbar[t][z]$. This shows that indeed, $\tilde{f}_t\in\Qbar[t][z]$, which concludes our proof of Proposition~\ref{prop:poly}.
\end{proof}


\section{The converse implication in Theorem~\ref{thm:main}}
\label{sec:converse}

We work with the notation as before, i.e., we have a family $f_t(z)\in\Qbar[t][z]$ of polynomials in normal form of degree $d\ge 2$, i.e.,
\begin{equation}
\label{eq:0000}
f_t(z)=z^d+\sum_{i=0}^{d-2} c_i(t)\cdot z^i\text{ with $c_i(t)\in\Qbar[t]$ for each $i=0,\dots, d-2$.}
\end{equation}
Also, we let $\alpha_1,\alpha_2,\beta\in\Qbar[t]$. 
We  work with a \emph{subset} of the general hypotheses from Theorem~\ref{thm:main}, i.e.,
\begin{equation}
\label{eq:III}
\text{$c_0(t)$ is not identically equal to $0$;}
\end{equation}
\begin{equation}
\label{eq:I}
\text{not all polynomials $c_0(t),\dots, c_{d-2}(t),\alpha_1(t),\alpha_2(t), \beta(t)$ are constant}
\end{equation}
and also,
\begin{equation}
\label{eq:II}
\text{ for each $i=1,2$, $\alpha_i$ is not persistent preperiodic under the action of $f_t(z)$.}
\end{equation}

The  main result of this Section is the following.
\begin{proposition}
\label{prop:converse}
Let $f_t(z),\alpha_1(t),\alpha_2(t),\beta(t)$ satisfy the conditions \eqref{eq:0000}, \eqref{eq:III}, \eqref{eq:I} and \eqref{eq:II}. Furthermore, assume that at least one of the following four conditions holds: 
\begin{itemize}
\item[(A)] there exists $\ell\in\N$ and there exists $i\in\{1,2\}$ such that 
\begin{equation}
\label{eq:1-22}
f_t^\ell\left(\alpha_i(t)\right)=\beta(t).
\end{equation} 
\item[(B)] there exist positive integers $\ell_1,\ell_2$  such that
\begin{equation}
\label{eq:2-22}
f_t^{\ell_1}\left(\alpha_1(t)\right) =f_t^{\ell_2}\left(\alpha_2(t)\right).
\end{equation}
\item[(C)] there exists $\tilde{f}_t\in\Qbar[t][z]$ and there exist integers $k\ge 2$ and $\ell_1,\ell_2,\ell_3\ge 1$   such that
\begin{equation}
\label{eq:C1-22}
f_t=\tilde{f}_t^k\text{ and also,} 
\end{equation}
\begin{equation}
\label{eq:C2-22}
\tilde{f}_t^{\ell_1}(\alpha_1(t))= \tilde{f}_t^{\ell_2}(\alpha_2(t))= \tilde{f}_t^{\ell_3}(\beta(t)).
\end{equation}
\item[(D)] there exists $\tilde{f}_t\in\Qbar[t][z]$ and there exist integers $k\ge 2$ and $\ell_1,\ell_2\ge 1$   such that 
\begin{equation}
\label{eq:D1-3}
f_t=\tilde{f}_t^k\text{ and also, we have} 
\end{equation}
\begin{equation}
\label{eq:D2-3}
\tilde{f}_t^{\ell_1}(\alpha_1(t))=\tilde{f}_t^{\ell_2}(\alpha_2(t)).
\end{equation}
Furthermore, there exists $N_0\in\N$ satisfying the following two conditions:
\begin{equation}
\label{eq:D4-3}
\gcd(N_0,k)\mid \ell_2-\ell_1\text{, and }
\end{equation} 
\begin{equation}
\label{eq:D3-3}
\tilde{f}_t^{N_0}(\beta(t))=\beta(t).
\end{equation}
\end{itemize} 
Then the set  
\begin{equation}
\label{eq:converse_condition}
C_f(\alpha_1,\alpha_2;\beta)=\left\{\l\in\Qbar\colon \exists m,n\in\N\text{ such that }
f_\l^m(\alpha_1(\l))=f_\l^n(\alpha_2(\l))=\beta(\l)\right\}
\end{equation}
is infinite.
\end{proposition}

The main ingredient in our proof of Proposition~\ref{prop:converse} is the following result.
\begin{proposition}
\label{prop:real_converse}
Let $\alpha(t),\beta(t)\in\Qbar[t]$ and let $f_t(z)\in\Qbar[t][z]$ be a family of polynomials of degree $d\ge 2$ of the form~\eqref{eq:0000}, i.e., 
$$f_t(z)=z^d+\sum_{i=0}^{d-2}c_i(t)\cdot z^i.$$
satisfying also the following assumptions:
\begin{equation}
\label{eq:III-2}
\text{not all polynomials $c_0(t),\dots, c_{d-2}(t)$ and $\beta(t)$ are  equal to $0$; and}
\end{equation} 
\begin{equation}
\label{eq:I-2}
\text{not all polynomials $c_0(t),\dots, c_{d-2}(t),\alpha(t),\beta(t)$  are constant.} 
\end{equation}
Furthermore, assume
\begin{equation}
\label{eq:II-2}
\text{$\alpha$ is not persistent preperiodic under the action of $f_t$.}
\end{equation}
Then there exist infinitely many $\l\in\Qbar$ such that for some suitable $m\in\N$ (depending on $\l$), we have that $f_\l^m(\alpha(\l))=\beta(\l)$.
\end{proposition}

In Section~\ref{subsec:converse_1}, we prove Proposition~\ref{prop:real_converse}. In Section~\ref{subsec:converse_2}, we derive Proposition~\ref{prop:converse} as a consequence of Proposition~\ref{prop:real_converse}.


\subsection{Proof of Proposition~\ref{prop:real_converse}}
\label{subsec:converse_1}

In this Section~\ref{subsec:converse_1}, we work with the notation of Proposition~\ref{prop:real_converse} for $f_t(z),\alpha(t),\beta(t)$ satisfying the hypotheses from \eqref{eq:0000}, \eqref{eq:III-2}, \eqref{eq:I-2} and \eqref{eq:II-2}. As an aside, we note that hypothesis \eqref{eq:III-2} is weaker than the corresponding hypothesis \eqref{eq:III} from Proposition~\ref{prop:converse}.

First, we deal with a very special case in Proposition~\ref{prop:real_converse}.
\begin{lemma}
\label{lem:almost_constant}
If the polynomials $c_0(t),\dots, c_{d-2}(t)$ and $\alpha(t)$ are all constant, then Proposition~\ref{prop:real_converse} holds.
\end{lemma}

\begin{proof}[Proof of Lemma~\ref{lem:almost_constant}.]
Combining the hypothesis of Lemma~\ref{lem:almost_constant} with the hypothesis~\eqref{eq:I-2}, we conclude that $\beta(t)$ is \emph{not} a constant polynomial. But then for \emph{each} $n\in\N$, we can find $\l_n\in\Qbar$ with the property that
\begin{equation}
\label{eq:140}
\beta(\l_n)=f_{\l_n}^n(\alpha(\l_n))=f_t^n(\alpha(t))
\end{equation}
(note that the right-hand side of \eqref{eq:140} is independent of $t$ due to the hypothesis of Lemma~\ref{lem:almost_constant}). On the other hand, the right-hand side of \eqref{eq:140} \emph{changes} with $n$ due to the hypothesis~\eqref{eq:II-2} of Proposition~\ref{prop:real_converse}. Therefore, as we vary $n$, we do find infinitely many distinct $\l\in\Qbar$ satisfying the desired conclusion in Proposition~\ref{prop:real_converse}.
\end{proof}

Lemma~\ref{lem:almost_constant} allows us to work from now on in the proof of Proposition~\ref{prop:real_converse} under the following hypothesis:
\begin{equation}
\label{eq:I-3}
\text{not all polynomials $c_0(t),\dots, c_{d-2}(t),\alpha(t)$ are constant;}
\end{equation}
we note that \eqref{eq:I-3} is slightly stronger than hypothesis~\eqref{eq:I-2} from Proposition~\ref{prop:real_converse}.

For each $n\in\N$, we let 
\begin{equation}
\label{eq:1111}
P_{n,\alpha}(t):=f_t^n\left(\alpha(t)\right)\in\Qbar[t].
\end{equation}
The following easy fact is a consequence of hypotheses \eqref{eq:I-3} and \eqref{eq:II-2}.
\begin{lemma}
\label{lem:not_isotrivial}
There exists $n_0\in\N$ such that 
\begin{equation}
\label{eq:high_degree}
\deg_t\left(P_{n_0,\alpha}(t)\right)>\max\left\{\deg_t(\beta(t)), \max_{i=0}^{d-2} \deg_t\left(c_i(t)\right)\right\}.
\end{equation}
\end{lemma}

\begin{proof}[Proof of Lemma~\ref{lem:not_isotrivial}.]
First, we note that if $c_0(t),\dots, c_{d-2}(t)$ are all constant polynomials, then the hypothesis~\eqref{eq:I-3} yields that $\alpha(t)$ is not a constant polynomial. Therefore, Lemma~\ref{lem:iterates} yields that 
\begin{equation}
\label{eq:141}
\deg_t\left(P_{n,\alpha}(t)\right)=\deg_t(\alpha(t))\cdot d^n\text{ for each $n\in\N$.}
\end{equation}
Equation~\eqref{eq:141} delivers the desired conclusion from Lemma~\ref{lem:not_isotrivial}. So, from now on, we assume in the proof of Lemma~\ref{lem:not_isotrivial} that 
\begin{equation}
\label{eq:I-4}
\text{not all polynomials $c_0(t),\dots, c_{d-2}(t)$ are constant.}
\end{equation}

Now, observe that under hypothesis~\eqref{eq:I-4}, $f_t(z)$ is not isotrivial, i.e., it is not conjugated to a polynomial in $\Qbar[z]$ (see Lemma~\ref{lem:iso}, coupled with the fact that not all polynomials $c_i(t)$ are constant, according to \eqref{eq:I-4}). So, using that the polynomial $f_t\in\Qbar[t][z]$ is not isotrivial, then Proposition~\ref{prop:nonzero_height} yields that $\deg_t\left(P_{n,\alpha}(t)\right)$ is unbounded as $n$ varies (because $\hhat_t(\alpha(t))>0$). In particular, there exists $n_0\in\N$ such that inequality~\eqref{eq:high_degree} holds, as desired.
\end{proof}

Let $n_0\in\N$ such that inequality~\eqref{eq:high_degree} is satisfied. We let $D_0\in\N$ be the degree of $P_{n_0,\alpha}(t)$ (as a polynomial in $t$) and we let $A\in\Qbar^\ast$ be the coefficient of $t^{D_0}$ in $P_{n_0,\alpha}(t)$. Lemma~\ref{lem:iterates} yields that 
\begin{equation}
\label{lem:deg_ind}
\text{for each $m\in\N$, we have that $\deg_t\left(P_{n_0+m,\alpha}(t)\right)=D_0\cdot d^m$.}
\end{equation}
Furthermore, an easy induction based on the recurrence relation
\begin{equation}
\label{eq:rec}
P_{n+1,\alpha}(t)=P_{n,\alpha}(t)^d + \sum_{i=0}^{d-2} c_i(t)\cdot P_{n,\alpha}(t)^i
\end{equation}
yields that 
\begin{equation}
\label{lem:deg_ind_2}
\text{for each $m\in\N$, the coefficient of $t^{D_0d^m}$ in $P_{n_0+m,\alpha}(t)$ is $A^{d^m}$.}
\end{equation}

We continue our proof of Proposition~\ref{prop:real_converse} and we argue by contradiction; hence, we assume there exist finitely many $\l_1,\dots, \l_s\in\Qbar$ (for some $s\in\N$) such that for each $m\in\N$, the only possible roots of the polynomial 
\begin{equation}
\label{eq:Q}
Q_{m,\alpha,\beta}(t):=P_{n_0+m,\alpha}(t)-\beta(t) 
\end{equation}
are $\l_1,\dots,\l_s$. Using again inequality~\eqref{eq:high_degree} along with equations \eqref{lem:deg_ind} and \eqref{lem:deg_ind_2} yields that for each $m\in\N$, there exist some nonnegative integers $e_{m,i}$ for $i=1,\dots, s$, such that
\begin{equation}
\label{eq:31}
Q_{m,\alpha,\beta}(t)=A^{d^m}\cdot \prod_{i=1}^s \left(t-\l_i\right)^{e_{m,i}}.
\end{equation}
Using the recurrence relation \eqref{eq:rec}, we have
\begin{equation}
\label{eq:42}
Q_{m+1,\alpha,\beta}(t)+\beta(t)=  \left(Q_{m,\alpha,\beta}(t)+\beta(t)\right)^d + \sum_{i=0}^{d-2} c_i(t)\cdot \left(Q_{m,\alpha,\beta}(t)+\beta(t)\right)^i.
\end{equation}
We let $X$ be the plane curve defined over $\Qbar(t)$  given by the equation (with respect to the pair of coordinates $(x,y)$ of $\A^2$):
\begin{equation}
\label{eq:43}
y+\beta(t)=\left(x+\beta(t)\right)^d + \sum_{i=0}^{d-2}c_i(t)\cdot \left(x+\beta(t)\right)^i.
\end{equation}
Equation \eqref{eq:42} yields that $X$  contains all points of the form $\left(Q_{m,\alpha,\beta}(t), Q_{m+1,\alpha,\beta}(t)\right)$ for each $m\in\N$. Furthermore, equation~\eqref{eq:31} yields that $X$ contains  infinitely many points from the subgroup $\Gamma$ of $\mathbb{G}_m^2\left(\Qbar(t)\right)$ spanned by $(A,1)$,  $(1,A)$, along with $(t-\l_i,1)$ and $(1,t-\l_i)$ for each $i=1,\dots, s$. Then the classical Mordell-Lang conjecture for $\mathbb{G}_m^2$ (proven by Laurent \cite{Laurent}) yields that $X$ is a coset of an $1$-dimensional subtorus of $\mathbb{G}_m^2$. In particular, it means that $X$ is given by an equation of the form
\begin{equation}
\label{eq:coset_tori}
x^ay^b=C\text{ for some $a,b\in\Z$ and $C\in\overline{\Q(t)}^\ast$.}
\end{equation}
Equation \eqref{eq:43} yields that the only possibility is for $b=1$,  $a=-d$ and $C=1$, i.e., the equation~\eqref{eq:42} must actually reduce to
\begin{equation}
\label{eq:44}
y=x^d.
\end{equation}
On the other hand, the only possibility for the equation~\eqref{eq:42} to reduce to equation~\eqref{eq:44} is when $\beta(t)=0$ and also,  $c_i(t)=0$ for each $i=0,\dots, d-2$. However, this contradicts hypothesis~\eqref{eq:III-2} that not all polynomials $c_i(t)$ (for $i=0,\dots, d-2$) along with polynomial $\beta(t)$ are equal to $0$. So, this contradiction means that equation~\eqref{eq:31} cannot hold, i.e.,  there must exist infinitely many $\l\in\Qbar$ such that for some suitable $n\in\N$ (depending on $\l$), we have $f_\l^n(\alpha(\l))=\beta(\l)$.

This concludes our proof of Proposition~\ref{prop:real_converse}. 


\subsection{Proof of Proposition~\ref{prop:converse}}
\label{subsec:converse_2}

We work under the hypotheses from Proposition~\ref{prop:converse} for $f_t(z),\alpha_1(t),\alpha_2(t),\beta(t)$, which satisfy \eqref{eq:0000}, \eqref{eq:III}, \eqref{eq:I} and \eqref{eq:II}. In order to prove the desired conclusion from Proposition~\ref{prop:converse}, we assume each time that one of the conditions~(A)-(D) hold  and then we prove that the set $C_f(\alpha_1,\alpha_2;\beta)$ from~\eqref{eq:converse_condition} must  be infinite.

\begin{lemma}
\label{lem:A-1-22}
If hypothesis~(A) is met, then the conclusion holds in Proposition~\ref{prop:converse}.  
\end{lemma}

\begin{proof}[Proof of Lemma~\ref{lem:A-1-22}.]
Without loss of generality, we assume equation~\eqref{eq:1-22} (from hypothesis~(A)) is met for $i=1$, i.e., there exists $m\in\N$ such that 
\begin{equation}
\label{eq:144}
f_t^m(\alpha_1(t))=\beta(t). 
\end{equation}
We claim that in this case,
\begin{equation}
\label{eq:143}
\text{not all polynomials $c_0(t),\dots,c_{d-2}(t), \alpha_2(t),\beta(t)$ are constant.}
\end{equation}
Indeed, if $c_0(t),\dots, c_{d-2}(t),\alpha_2(t),\beta(t)$ are constant, then \eqref{eq:144} yields that also $\alpha_1(t)$ must be constant, thus contradicting our hypothesis~\eqref{eq:I}. So, from now on, in Lemma~\ref{lem:A-1-22} we may assume \eqref{eq:143} holds. This allows us to apply Proposition~\ref{prop:real_converse} to starting point $\alpha_2$ and target point $\beta$ (also, note that $\alpha_2$  is not persistent preperiodic according to hypothesis~\eqref{eq:II}); thus, we conclude that there exist infinitely many $\l\in\Qbar$ along with some suitable positive integers $n$ (depending on $\l$) such that 
\begin{equation}
\label{eq:145}
f_\l^n(\alpha_2(\l))=\beta(\l).
\end{equation} 
Coupling equations~\eqref{eq:145} and \eqref{eq:144} allows us to conclude our proof of Lemma~\ref{lem:A-1-22}.
\end{proof}

\begin{lemma}
\label{lem:B-1-22}
If hypothesis~(B) is met, then the conclusion holds in Proposition~\ref{prop:converse}.  
\end{lemma}

\begin{proof}[Proof of Lemma~\ref{lem:B-1-22}.]
So, we assume there exist positive integers $\ell_1,\ell_2$ such that 
\begin{equation}
\label{eq:51}
f_t^{\ell_1}(\alpha_1(t))=f_t^{\ell_2}(\alpha_2(t)). 
\end{equation}

Let $\ell:=\max\{\ell_1,\ell_2\}$. Since $\alpha_1$ is not persistent preperiodic for $f_t$ (according to hypothesis~\eqref{eq:II}), then also $f_t^\ell(\alpha_1(t))$ is not persistent preperiodic under the action of $f_t$. Furthermore, we claim that
\begin{equation}
\label{eq:146}
\text{not all polynomials $c_0(t),\dots, c_{d-2}(t),f_t^\ell(\alpha_1(t)),\beta(t)$ are constant}
\end{equation}
Indeed, if all polynomials $c_0(t),\dots, c_{d-2}(t),f_t^\ell(\alpha_1(t))$ were constant, then immediately we get that $\alpha_1(t)$ is also constant. Furthermore,  equation~\eqref{eq:51} yields that also $\alpha_2(t)$ is constant. So, then assuming that $\beta(t)$ is also constant would lead to a contradiction to hypothesis~\eqref{eq:I}. Therefore, \eqref{eq:146} must be valid. This allows us to apply Proposition~\ref{prop:real_converse} to starting point $f_t^\ell(\alpha_1(t))$ and target point $\beta(t)$ and thus, we get that there exist infinitely many $\l\in\Qbar$ along with suitable $m\in\N$ (depending on $\l$) such that 
\begin{equation}
\label{eq:50}
f_\l^{m}\left(f_\l^\ell(\alpha_1(\l))\right)=\beta(\l).
\end{equation}
Combining equations \eqref{eq:50} and \eqref{eq:51} yields that also 
\begin{equation}
\label{eq:52}
f_\l^{m+\ell+\ell_2-\ell_1}(\alpha_2(\l))=\beta(\l).
\end{equation}
Equations \eqref{eq:50} and \eqref{eq:52} provide the desired conclusion in Lemma~\ref{lem:B-1-22}.
\end{proof}

\begin{lemma}
\label{lem:C-1-22}
If hypothesis~(C) is met, then the conclusion holds in Proposition~\ref{prop:converse}.  
\end{lemma}

\begin{proof}[Proof of Lemma~\ref{lem:C-1-22}.]
So, according to equation~\eqref{eq:C1-22}, there exists a polynomial $\tilde{f}_t\in\Qbar[t][z]$ and an integer $k\ge 2$ such that 
\begin{equation}
\label{eq:eq:C}
\tilde{f}_t^k=f_t. 
\end{equation}
Furthermore, hypothesis~\eqref{eq:C2-22} yields that for some positive integers $\ell_1,\ell_2,\ell_3$, we have
\begin{equation}
\label{eq:eq:C2-22}
\tilde{f}_t^{\ell_1}(\alpha_1(t))=\tilde{f}_t^{\ell_2}(\alpha_2(t))= \tilde{f}_t^{\ell_3}(\beta(t)).
\end{equation}
Since we know that $\alpha_1,\alpha_2$ are not persistent preperiodic under the action of $f_t$ (see hypothesis~\eqref{eq:II}), then equation~\eqref{eq:eq:C} yields that $\beta$  is  not persistent preperiodic under the action of $f_t=\tilde{f}_t^k$; furthermore, $\tilde{f}_t(\beta(t))$ is also not persistent preperiodic under the action of $f_t$ (see again \eqref{eq:eq:C}). Moreover, we claim that
\begin{equation}
\label{eq:147}
\text{not all polynomials $c_0(t),\dots, c_{d-2}(t),\tilde{f}_t(\beta(t)),\beta(t)$ are constant.}
\end{equation}
Indeed, if we were to assume that all polynomials $c_0(t),\dots, c_{d-2}(t)$ are constant, then we would have that $\tilde{f}_t\in\Qbar[z]$ (since $f_t=\tilde{f}_t^k$). So, if also $\beta(t)$ is a constant polynomial, then equation~\eqref{eq:eq:C2-22} yields that both $\alpha_1(t)$ and $\alpha_2(t)$ would also have to be constant polynomials. But then we would have that $c_0(t),\dots, c_{d-2}(t)$, $\alpha_1(t)$, $\alpha_2(t)$ and $\beta(t)$ are all constant polynomials, which contradicts hypothesis~\eqref{eq:I}. Hence, we must have that equation~\eqref{eq:147} is valid. This allows us to  apply Proposition~\ref{prop:real_converse} to the dynamical system formed by the polynomial family $f_t(z)$, starting point $\tilde{f}_t(\beta(t))$ and target point $\beta(t)$. Therefore, we  conclude the existence of infinitely many $\l\in\Qbar$ along with a suitable positive integers $s$ (depending on $\l$) such that
\begin{equation}
\label{eq:eq:0000}
\beta(\l)=f_\l^s\left(\tilde{f}_\l(\beta(\l))\right)= \tilde{f}_\l^{s k+1}(\beta(\l)).
\end{equation}
So, let $\l\in\Qbar$ and let $s\in\N$ as in equation \eqref{eq:eq:0000}; therefore, $\beta(\l)$ is periodic under the action of $\tilde{f}_\l$ of minimal period $r$ (depending on $\l$) dividing $ks+1$. In particular, $r$ is coprime with $k$; combining this information with  \eqref{eq:eq:C2-22}, we get that for any $q_1,q_2,q_3\in\N$, we have
\begin{equation}
\label{eq:60}
\tilde{f}_\l^{\ell_1+q_3+q_1r}(\alpha_1(\l))= \tilde{f}_\l^{\ell_2+q_3+q_2r}(\alpha_2(\l))= \tilde{f}_\l^{\ell_3 + q_3}(\beta(\l)).
\end{equation}
Because $\gcd(r,k)=1$, equation~\eqref{eq:60} yields that we can find suitable $q_1,q_2,q_3\in\N$ and there exist positive integers $m,n,q_4$ with the property that
\begin{equation}
\label{eq:61}
\ell_1+q_3+q_1r=k\cdot m\text{, }\ell_2+q_3+q_2r=k\cdot n\text{ and }\ell_3 + q_3 = r q_4.
\end{equation}
Finally, combining equations~\eqref{eq:61} and \eqref{eq:60}, along with the fact that $\tilde{f}_t^k=f_t$ (and that $\beta(\l)$ has period $r$ under the action of $\tilde{f}_\l$), we get
$$f_\l^m(\alpha_1(\l))=f_\l^n(\alpha_2(\l))=\beta(\l).$$
This concludes our proof of Lemma~\ref{lem:C-1-22}.
\end{proof}

\begin{lemma}
\label{lem:D-1-22}
If hypothesis~(D) is met, then the conclusion holds in Proposition~\ref{prop:converse}.  
\end{lemma}

\begin{proof}[Proof of Lemma~\ref{lem:D-1-22}.]
So, according to hypotheses \eqref{eq:D1-3} and \eqref{eq:D2-3}, there exists some $\tilde{f}_t(z)\in\Qbar[t][z]$ and there exists some integer $k\ge 2$ such that $f_t=\tilde{f}_t^k$. Also, there exist positive integers $\ell_1$ and $\ell_2$ such that
\begin{equation}
\label{eq:112}
\tilde{f}_t^{\ell_1}(\alpha_1(t))=\tilde{f}_t^{\ell_2}(\alpha_2(t)).
\end{equation}
Without loss of generality, we assume $\ell_2\ge \ell_1$. We claim that 
\begin{equation}
\label{eq:148}
\text{not all polynomials $c_0(t),\dots, c_{d-2}(t),\alpha_1(t),\beta(t)$ are constant.}
\end{equation}
Indeed, if $c_0(t),\dots, c_{d-2}(t)$ are constant polynomials, then $f_t(z)\in\Qbar[z]$ and so, also $\tilde{f}_t(z)\in\Qbar[z]$. Then assuming also that $\alpha_1(t),\beta(t)$ are constant polynomials, equation~\eqref{eq:112} yields that $\alpha_2(t)$ is a constant polynomial as well. But then we would have that $c_0(t),\dots, c_{d-2}(t)$, $\alpha_1(t)$, $\alpha_2(t)$ and $\beta(t)$ are constant polynomials, which contradicts hypothesis~\eqref{eq:I}. So, indeed, we must have that equation~\eqref{eq:148} is valid. 

Using \eqref{eq:148} along with hypothesis~\eqref{eq:II}, we can apply Proposition~\ref{prop:real_converse} to starting point $\alpha_1$ (which is not persistent preperiodic under the action of $f_t$) and target point $\beta$; thus, we get the existence of infinitely many $\l\in\Qbar$ along with suitable positive integers $m_\l$ such that
\begin{equation}
\label{eq:114}
f_\l^{m_\l}(\alpha_1(\l))=\beta(\l)\text{ and therefore, } \tilde{f}_{\l}^{km_\l}(\alpha_1(\l))=\beta(\l).
\end{equation}
Moreover, we may assume $m_\l>\ell_1$ in equation~\eqref{eq:114} since for each of the finitely many integers $i\in\{1,\dots, \ell_1\}$, there are only finitely many $\l\in\Qbar$ such that $\tilde{f}_\l^i(\alpha_1(\l))=\beta(\l)$. Indeed,  otherwise, we would get that for a suitable $i_0\in\{1,\dots, \ell_1\}$, we have that $\tilde{f}_t^{i_0}(\alpha_1(t))=\beta(t)$, which contradicts the hypothesis that $\alpha_1$ is not persistent preperiodic (note that $\beta$ is persistent periodic, according to equation~\eqref{eq:D3-3}).

So, using that $m_\l>\ell_1$, equations \eqref{eq:114} and \eqref{eq:112} yield that
\begin{equation}
\label{eq:115}
\beta(\l)=\tilde{f}_\l^{km_\l}(\alpha_1(\l))= \tilde{f}_\l^{km_\l-\ell_1+\ell_2}(\alpha_2(\l)).
\end{equation}
But then divisibility~\eqref{eq:D4-3} yields the existence of some suitable positive integer $i_\l>\ell_2$ such that
\begin{equation}
\label{eq:116}
N_0\mid ki_\l+\ell_1-\ell_2.
\end{equation}
We let $j_\l:=ki_\l+\ell_1-\ell_2$; then $j_\l$ is a positive integer because $i_\l>\ell_2$. Pre-composing the third part of  equation~\eqref{eq:115}  by $\tilde{f}_\l^{j_\l}$ and using both equation~\eqref{eq:116} and the fact that $\beta(\l)$ has period dividing $N_0$ due to hypothesis~\eqref{eq:D3-3}, we conclude that
\begin{equation}
\label{eq:117}
\beta(\l)=\tilde{f}_\l^{km_\l-\ell_1+\ell_2+j_\l}(\alpha_2(\l)).
\end{equation}
Using equation \eqref{eq:117} along with the fact that $j_\l=ki_\l+\ell_1-\ell_2$, we obtain that $\tilde{f}^{k(m_\l+i_\l)}(\alpha_2(\l))=\beta(\l)$. Coupling this with \eqref{eq:D1-3}, we have that
\begin{equation}
\label{eq:118}
f_\l^{m_\l+i_\l}(\alpha_2(\l))=\beta(\l).
\end{equation}
Equations \eqref{eq:115} and \eqref{eq:118} deliver the desired conclusion from Proposition~\ref{prop:converse}, as claimed in Lemma~\ref{lem:D-1-22}.
\end{proof}

Combining Lemmas~\ref{lem:A-1-22}, \ref{lem:B-1-22}, \ref{lem:C-1-22} and \ref{lem:D-1-22}, we obtain the desired conclusion in Proposition~\ref{prop:converse}, which provides the converse implication  from Theorem~\ref{thm:main}. 

The remaining sections of our paper will be used for proving the direct implication from Theorem~\ref{thm:main}; we will start with setting up the height machine in Section~\ref{sec:heights}.


\section{Heights}
\label{sec:heights}

Throughout Section~\ref{sec:heights}, we work with a family of polynomials $f_t(z)\in\Qbar[t][z]$ as in equation~\eqref{eq:-1}, i.e.,
\begin{equation}
\label{eq:eq:-1}
f_t(z)=z^d+\sum_{i=0}^{d-2} c_i(t)\cdot z^i\text{, where $c_i(t)\in\Qbar[t]$ for $i=0,\dots, d-2$.}
\end{equation}
We fix from now on a number field $L$ containing all coefficients of each $c_i(t)$. We start by defining the Weil height on $\Qbar$.


\subsection{The Weil height}
\label{subsec:Weil}

There exists a set $\Omega_L$ of absolute values on $L$ (extending the usual $p$-adic absolute values and the archimedean norm from $\Q$) along with suitable positive integers $n_v$ for each $v\in\Omega_L$ such that for each $\gamma\in L^\ast$, the following product formula holds:
\begin{equation}
\label{eq:70}
\prod_{v\in\Omega_L}|\gamma|_v^{n_v}=1.
\end{equation}
For each $v\in\Omega_L$, we fix an extension of the absolute value $|\cdot |_v$ to $\Qbar$. 
Then the Weil height $h(\cdot):=h_L(\cdot )$ is defined as follows. For each $\gamma\in\Qbar$, we let
\begin{equation}
\label{eq:71}
h(\gamma)=\frac{1}{[L(\gamma):L]}\cdot \sum_{v\in\Omega_L}n_v\cdot \sum_{\substack{\sigma:L(\gamma)\lra\Qbar\\ \sigma|_L={\rm id}|_L}} \log^+\left|\sigma(\gamma)\right|_v,
\end{equation}
where for each real number $a$, we use the notation $\log^+a:=\log\max\{1,a\}$.

Next, we define the canonical heights corresponding to the polynomials from the family of polynomials $\{f_\l(z)\}_{\l\in\Qbar}$.


\subsection{Canonical heights in our family of polynomials}
\label{subsec:canonical}

Let $v\in\Omega_L$ and let $\l\in\Qbar$.

For each $\gamma\in\Qbar$, we define the local canonical height of $\gamma$ with respect  to the polynomial $f_\l(z)$ from our family of polynomials~\eqref{eq:eq:-1}, as follows:
\begin{equation}
\label{eq:local_canonical}
\hhat_{\l,v}(\gamma):=\lim_{n\to\infty} \frac{\log^+\left|f_\l^n(\gamma)\right|_v}{d^n}.
\end{equation}
The global canonical height of $\gamma$ with respect to $f_\l$ is defined as follows:
\begin{equation}
\label{eq:global_canonical}
\hhat_\l(\gamma):=\lim_{n\to\infty} \frac{h\left(f_\l^n(\gamma)\right)}{d^n}.
\end{equation}
Using the fact that for given $\l,\gamma\in\Qbar$, we have that $\hhat_{\l,v}(\gamma)=0$ for all but finitely many places $v\in\Omega_L$,  then combining equations~\eqref{eq:global_canonical}, \eqref{eq:local_canonical} and \eqref{eq:71}, we get
\begin{equation}
\label{eq:72}
\hhat_\l(\gamma)= \frac{1}{[L(\lambda,\gamma):L]}\cdot \sum_{v\in\Omega_L} n_v\cdot \sum_{\substack{\sigma: L(\lambda,\gamma)\lra\Qbar\\ \sigma|_L={\rm id}|_L}} \hhat_{\sigma(\l),v}\left(\sigma(\gamma)\right).
\end{equation}
We also note the following  standard result (see \cite[Theorem~3.22]{Silverman}).
\begin{lemma}
\label{lem:clearly}
For each $\l\in\Qbar$ and for each $\gamma\in\Qbar$, we have that $\gamma$ is preperiodic under the action of $f_\l(z)$ if and only if $\hhat_\l(\gamma)=0$.
\end{lemma}
For more details regarding the local and global canonical heights associated to polynomials, see \cite{Call-Silverman, Silverman}. 


\subsection{Variation of the canonical height}
\label{subsec:variation}

We continue with our notation from Section~\ref{subsec:canonical} for the canonical heights $\hhat_\l$ constructed with respect to the polynomials $f_\l(z)$ from the family of polynomials \eqref{eq:eq:-1}. We recall that $L$ is a given number field containing the coefficients of each $c_i(t)$ for $i=0,\dots, d-2$, where $f_t(z)=z^d+\sum_{i=0}^{d-2}c_i(t)\cdot z^i$. 

Now, let $\alpha\in L[t]$. We recall the definition of $P_{n,\alpha}(t)=f_t^n(\alpha(t))$ (see also equation~\eqref{eq:1111}) and then recall the definition of the canonical height of $\alpha(t)$ with respect to $f_t(z)$: 
\begin{equation}
\label{eq:73}
\hhat_t(\alpha(t)):=\lim_{n\to\infty} \frac{\deg_t\left(P_{n,\alpha}(t)\right)}{d^n}.
\end{equation} 
The following important result is a special case of  \cite[Theorem~1]{Patrick-Crelle}.

\begin{proposition}
\label{prop:useful_variation}
With the above notation for $f_t(z)$ as in \eqref{eq:eq:-1}, $L$,  $\hhat_\l$ (for each $\l\in\Qbar$) as in \eqref{eq:global_canonical}, $\hhat_t$ as in \eqref{eq:73}, we let $\alpha\in L[t]$. Then there exists a positive real number $B$ with the property that for each $\l\in\Qbar$, we have: 
\begin{equation}
\label{eq:the_inequality}
\left|\hhat_\l(\alpha(\l))-\hhat_t(\alpha(t))\cdot h(\l)\right|\le B.
\end{equation}
\end{proposition}

\begin{remark}
\label{rem:variation}
Proposition~\ref{prop:useful_variation} is part of a series of results regarding the variation of the canonical height
in an algebraic family of maps. Variants of Proposition~\ref{prop:useful_variation} were proven first over number fields
in the case of elliptic curves (see \cite{Tate}), and then extended to all abelian varieties (see \cite{Sil83}). Similar results were then proven for polynomial families (see \cite{Call-Silverman, Patrick-Crelle}) and also for certain families of rational functions (see \cite{G-M}).
\end{remark}

Using Proposition~\ref{prop:useful_variation}, we obtain a key result for our proof of the direct implication in Theorem~\ref{thm:main}. 


\subsection{A useful result regarding points of small heights}
\label{subsec:small_heights}
We prove the following.

\begin{proposition}
\label{prop:heights}
Let $L$ be a number field, let $h=h_L$ be its Weil height constructed as in \eqref{eq:71} and let $f_t(z)\in L[t][z]$ be a family of polynomials of degree $d\ge 2$ as in \eqref{eq:eq:-1}. We let $\alpha,\beta\in L[t]$ and also, for each $\l\in\Qbar$, let  $\hhat_\l$ be the canonical heights constructed as in \eqref{eq:global_canonical}. Assume $\alpha$ is an active point for $f_t$, i.e., $\hhat_t(\alpha(t))>0$. Furthermore, assume there exists a sequence $\{\l_k\}_{k\ge 1}$ of elements of $\Qbar$ along with a sequence $\{m_k\}_{k\ge 1}$ of positive integers such that 
\begin{equation}
\label{eq:75}
\lim_{k\to\infty}m_k=\infty 
\end{equation}
with the property that
\begin{equation}
\label{eq:20}
f_\l^{m_k}\left(\alpha(\l_k)\right)=\beta(\l_k)\text{ for each }k\ge 1.
\end{equation}
Then the following two conclusions must hold:
\begin{itemize}
\item[(1)] there exists a positive real number $C$ such that
\begin{equation}
\label{eq:160}
h(\l_k)\le C\text{ for each $k\in\N$.}
\end{equation}
\item[(2)] $\lim_{k\to\infty} \hhat_{\l_k}\left(\alpha(\l_k)\right)=0$.
\end{itemize}
\end{proposition}

\begin{proof}
Proposition~\ref{prop:useful_variation} yields that there exist  real numbers $A_1,B_1\ge 0$ and $A_2,B_2>0$ with the property that 
\begin{equation}
\label{eq:22}
\left|\hhat_{\l}\left(\alpha(\l)\right)-A_1h(\l)\right|\le A_2\text{ for each $\l\in\Qbar$}
\end{equation}
and also,
\begin{equation}
\label{eq:23}
\left|\hhat_{\l}\left(\beta(\l)\right)-B_1h(\l)\right|\le B_2\text{ for each $\l\in\Qbar$.}
\end{equation}
Furthermore, Proposition~\ref{prop:useful_variation} coupled with the fact that $\alpha$ is an active point for $f_t$ yields that $A_1>0$ as well. On the other hand, equation~\eqref{eq:20} yields that 
\begin{equation}
\label{eq:24}
d^{m_k}\cdot \hhat_{\l_k}\left(\alpha(\l_k)\right)= \hhat_{\l_k}\left(\beta(\l_k)\right)\text{ for each $k\ge 1$.}
\end{equation}
Combining equations \eqref{eq:22}, \eqref{eq:23} and \eqref{eq:24} yields that for each $k\ge 1$, we have
\begin{equation}
\label{eq:25}
d^{m_k}\cdot \left(A_1\cdot h(\l_k)-A_2\right)\le B_1\cdot h(\l_k)+B_2;
\end{equation}
equation~\eqref{eq:25} immediately yields the inequality
\begin{equation}
\label{eq:26}
h(\l_k)\le \frac{d^{m_k}\cdot A_2+B_2}{d^{m_k}\cdot A_1 - B_1}\text{ for each $k\ge 1$.}
\end{equation}
Using our hypothesis that $\lim_{k\to\infty}m_k=\infty$, we conclude that there exists a positive real number $C$ such that 
\begin{equation}
\label{eq:27}
h(\l_k)\le C\text{ for each $k\ge 1$,}
\end{equation}
as desired in conclusion~(1) of Proposition~\ref{prop:heights} (see \eqref{eq:160}). 

Next, combining inequalities \eqref{eq:27} and \eqref{eq:23}, we get that
\begin{equation}
\label{eq:29}
\hhat_{\l_k}(\beta(\l_k))\le B_1C+B_2\text{ for each $k\ge 1$.}
\end{equation}
But then equations \eqref{eq:24} and \eqref{eq:29} yield that 
\begin{equation}
\label{eq:28}
\hhat_{\l_k}\left(\alpha(\l_k)\right)\le \frac{B_1C+B_2}{d^{m_k}}\text{ for each $k\ge 1$.} 
\end{equation}
Using inequality~\eqref{eq:28} along with the hypothesis that $\lim_{k\to\infty}m_k=\infty$ allows us to obtain the desired conclusion~(2). This concludes our proof for Proposition~\ref{prop:heights}.
\end{proof}

Next, in Section~\ref{subsec:constant_family}, we derive another useful result regarding heights.


\subsection{The special case of a constant family of polynomials}
\label{subsec:constant_family}

Throughout Section~\ref{subsec:constant_family}, we work with a constant family of polynomials, i.e., for each $i=0,\dots, d-2$, the polynomials $c_i(t)$ are constant in $\Qbar$. So, our family of polynomials $f_t(z)$ is of the form
\begin{equation}
\label{eq:161}
f_t(z):=z^d+\sum_{i=0}^{d-2}c_i z^i\text{, where each $c_i\in\Qbar$.}
\end{equation}
In particular, this means that for each $\l\in\Qbar$, we have
\begin{equation}
\label{eq:162}
f_\l(z)=f_t(z)=z^d+\sum_{i=0}^{d-2}c_i z^i.
\end{equation}
Therefore, for the sake of simplifying our notation, we write $f:=f_t=f_\l$ for each $\l\in\Qbar$. In particular, for any $\l,\gamma\in\Qbar$, the canonical height of $\gamma$ with respect to $f_\l(z)$ is simply the canonical height of $\gamma$ with respect to the polynomial $f(z)$ and it is denoted by $\hhat_f(\gamma)$, i.e.,
\begin{equation}
\label{eq:230}
\hhat_f(\gamma)=\lim_{n\to\infty} \frac{h\left(f^n(\gamma)\right)}{d^n}.
\end{equation}
The following result deals with the special case in the direct implication from Theorem~\ref{thm:main} when both the family of polynomials is constant and precisely one of the two starting points is also constant.

\begin{proposition}
\label{prop:constant_family}
With the above notation for $f(z)=z^d+\sum_{i=0}^{d-2}c_iz^i$ as in \eqref{eq:161} and \eqref{eq:162}, let $\alpha_1\in\Qbar$ and let $\alpha_2(t),\beta(t)\in\Qbar[t]$. Assume the following three conditions are met:
\begin{itemize}
\item[(i)] $\alpha_1$ is not preperiodic under the action of $f(z)$; 
\item[(ii)] $\min\left\{\deg_t(\alpha_2(t)), \deg_t(\beta(t))\right\}\ge 1$;  and 
\item[(iii)] the following set is infinite:
\begin{equation}
\label{eq:163}
C_f(\alpha_1,\alpha_2;\beta):=\left\{\l\in\Qbar\colon f^m(\alpha_1)=f^n(\alpha_2(\l))=\beta(\l)\text{ for some $m,n\in\N$}\right\}.
\end{equation}
\end{itemize}
Then there exists $n\in\N$ such that $f^n(\alpha_2(t))=\beta(t)$.
\end{proposition}

\begin{proof}
Using hypothesis~(i), we get that 
\begin{equation}
\label{eq:171}
A:=\hhat_f(\alpha_1)>0.
\end{equation} 
Hypothesis~(ii) yields that for each $m\in\N$, we have that 
\begin{equation}
\label{eq:164}
\text{there exist finitely many $\l\in\Qbar$ such that $f^m(\alpha_1) = \beta(\l)$.}
\end{equation}

We argue by contradiction and therefore, we assume that for each $n\in\N$,
\begin{equation}
\label{eq:165}
\text{there exist finitely many $\l\in\Qbar$ such that $f^n(\alpha_2(\l))=\beta(\l)$.}
\end{equation}
Therefore, combining \eqref{eq:164} and \eqref{eq:165} with our hypothesis~(iii), we get that there exists an infinite sequence $\{\l_k\}_{k\in\N}$ of algebraic numbers, along with two sequences $\{m_k\}_{k\in\N}$ and $\{n_k\}_{k\in\N}$ of positive integers such that the following conditions are met:
\begin{equation}
\label{eq:166}
f^{m_k}(\alpha_1)=f^{n_k}(\alpha_2(\l_k))=\beta(\l_k)\text{ for each $k\in\N$, and}
\end{equation}
\begin{equation}
\label{eq:167}
\lim_{k\to\infty} m_k=\lim_{k\to\infty}n_k=\infty.
\end{equation}
Since $\deg_t(\alpha_2(t))\ge 1$ and $f_t(z)=f(z)$ is a constant family of polynomials, then we get that $\alpha_2$ is an active point for $f_t$. Therefore, we may apply Proposition~\ref{prop:heights} to the constant family of polynomials $f_t=f$ along with the active starting point $\alpha_2$ and target point $\beta$; this allows us to conclude that there exists a positive real constant $C$ such that 
\begin{equation}
\label{eq:168}
h(\l_k)\le C\text{ for each $k\in\N$.}
\end{equation}
Inequality \eqref{eq:168} coupled with \cite[Theorem~3.11]{Silverman} yields  that there exists another positive real constant $B$ with the property that
\begin{equation}
\label{eq:169}
h\left(\beta(\l_k)\right)\le B\text{ for each $k\in\N$.}
\end{equation}
But then combining \eqref{eq:166} and \eqref{eq:169}, we get that
\begin{equation}
\label{eq:170}
h\left(f^{m_k}(\alpha_1)\right)\le B\text{ for each $k\in\N$.}
\end{equation} 
Now, using \cite[Theorem~3.20]{Silverman} coupled with inequality \eqref{eq:170}, we get that there exists another positive real constant $B'$ such that for \emph{all} $k\in\N$, we have
\begin{equation}
\label{eq:271}
\hhat_f\left(f^{m_k}(\alpha_1)\right)\le B'.
\end{equation}
Using \eqref{eq:171} and \eqref{eq:271}, we obtain that
\begin{equation}
\label{eq:172}
d^{m_k}\le \frac{B'}{A}\text{ for each $k\in\N$.}
\end{equation}
However, inequality~\eqref{eq:172} contradicts \eqref{eq:167}, thus showing that our assumption~\eqref{eq:165} is not valid. Hence, we must have that $f^n(\alpha_2(t))=\beta(t)$ for some positive integer $n$.

This concludes our proof of Proposition~\ref{prop:constant_family}.
\end{proof}


\section{Points of small height}
\label{subsec:Berkovich}

Let $L$ be a number field. We work with the  family of polynomials $f_t(z)\in L[t][z]$ as in \eqref{eq:eq:-1}, i.e., for each $\l\in\Qbar$, we have 
\begin{equation}
\label{eq:80}
f_\l(z)=z^d+\sum_{i=0}^{d-2}c_i(\l)\cdot z^i\text{, where $c_i(t)\in L[t]$.}
\end{equation}
We recall the definition of the canonical height from \eqref{eq:global_canonical}, i.e., for each $\l\in\Qbar$ and for each $\gamma\in\Qbar$, we have
\begin{equation}
\label{eq:234}
\hhat_\l(\gamma)=\lim_{n\to\infty} \frac{h\left(f_\l^n(\gamma)\right)}{d^n}.
\end{equation}
We let $\Omega:=\Omega_L$ be the set of absolute values on $L$ (see also Section~\ref{subsec:Weil}). For each $i=1,2$, we let $\alpha_i(t)\in L[t]$ and assume  
\begin{equation}
\label{eq:232}
\text{$\alpha_i$ is an active point for $f_t(z)$ for $i=1,2$.}
\end{equation} 
The main goal of Section~\ref{subsec:Berkovich} is the following.
\begin{theorem}
\label{thm:iff}
Let $L$, $f_\l$ as in \eqref{eq:80}, let $\hhat_{\l}$ (for each $\l\in\Qbar$) as in \eqref{eq:234} and let $\alpha_i\in L[t]$ (for $i=1,2$) satisfying the assumptions from \eqref{eq:232}.    Furthermore,  
assume there exists an infinite sequence  $\{\lambda_n\}$ in $\Qbar$ with the property that 
\begin{equation}
\label{eq:Bogomolov}
\lim_{n\to\infty} \hhat_{\l_n}\left(\alpha_1(\l_n)\right)=\lim_{n\to\infty} \hhat_{\l_n}\left(\alpha_2(\l_n)\right)=0.
\end{equation}
Then for each $\l\in\Qbar$, we have that $\alpha_1(\l)$ is preperiodic under the action of $f_\l(z)$ if and only if $\alpha_2(\l)$ is preperiodic under the action of $f_\l(z)$.
\end{theorem}
Variants of this result were proven for families of polynomials (see  \cite{Matt, Matt-2, FG, GHT-ANT}), including over fields of positive characteristic (see \cite{1, AA}); since the steps are similar in all these results, we will sketch the proof of Theorem~\ref{thm:iff} and refer to the aforementioned papers for the more technical details.


\subsection{Setup for the proof of Theorem~\ref{thm:iff}}
\label{subsec:note}

With the notation as in Theorem~\ref{thm:iff}, we define (as before) for each $i=1,2$:
\begin{equation}
\label{eq:P}
P_{n,\alpha_i}(t):=f^n_t(\alpha_i(t))\text{ for each 
$n\in\N$;} 
\end{equation}
then $P_{n,\alpha_i}(t)$ is a polynomial in $t$.  Since we assume that $\alpha_i$ (for $i=1,2$) is active for $f_t(z)$,  Lemma~\ref{lem:not_isotrivial} yields the existence of some $n_i\in\N$ (for $i=1,2$) such that
\begin{equation}
\label{eq:81}
\deg_t\left(P_{n_i,\alpha_i}(t)\right)> \max_{j=0}^{d-2}\deg_t\left(c_j(t)\right).
\end{equation}
Furthermore, Equation~\ref{lem:deg_ind} yields that for each $i=1,2$ and for each $m\ge 1$, we have
\begin{equation}
\label{eq:82}
\deg_t\left(P_{n_i+m,\alpha_i}(t)\right)=d^m\cdot \deg_t\left(P_{n_i,\alpha_i}(t)\right).
\end{equation}
We define (for the sake of simplifying our notation) the following:
\begin{equation}
\label{eq:83}
\text{$D_{i}:=\deg_t\left(P_{n_i,\alpha_i}(t)\right)$ for each $i=1,2$.}
\end{equation}

Now, for each $v\in\Omega$, we let $\C_v$ be an algebraically closed field,  which is also complete with respect to a fixed extension of $|\cdot |_v$ to $\C_v$; more precisely, $\C_v$ is the completion of an algebraic closure of the completion of $L$ at the place $v$. 

Let $\l\in \C_v$ and define (similar to \eqref{eq:local_canonical}) the local canonical height
$\hhat_{\lambda,v}(\gamma)$ of $\gamma\in \C_v$ with respect to the polynomial $f_\lambda$; more precisely, we have the formula 
\begin{equation}
\label{eq:def}
\hhat_{\lambda,v}(\gamma) :=\lim_{n\to\infty}
\frac{\log^+|f_\lambda^n(\gamma)|_v}{d^{n}}.
\end{equation} 
Then  $\hhat_{\lambda, v}(\gamma)$ is a continuous function in both $\l$
and $\gamma$ on $\C_v$ (for more details, see \cite{Matt, GHT-ANT, AA}). 

For each $v\in \Omega$ and $\l\in \C_v$, recall the construction from \cite{Matt} of the $v$-adic filled Julia set for $f_{\l}$ defined as follows:  
\[
  K_{\l,v} :=\{z\in \C_{v}\mid \hhat_{\lambda,v}(z) = 0\} = \{z\in \C_{v}\mid | f_{\l}^{n} (z) |_{v} \not\to \infty  \;\text{ as $n\to \infty$} \}.
\]

Let $i\in\{1,2\}$, let $v\in \Omega$ and  let $\Aberk{\C_v}$ denote the Berkovich affine  line over
$\C_v$ (see \cite{BR} or  \cite[Section~2]{Matt} for more details).  Then, the {\em generalized Mandelbrot set} $M_{\alpha_i,v}\subset \Aberk{\C_v}$ associated to $\alpha_i$ at $v$ is defined to be the closure in $\Aberk{\C_v}$ of the subset of $\C_v$ consisting of all parameters $\l\in \C_v$ such that the orbit of $\alpha_i(\l)$ is $v$-adically  bounded under the action of $f_{\l}(z)$ (see \cite{Matt, GHT-ANT}). Note that for such a parameter $\l$, since the orbit of $\alpha_i(\l)$ is $v$-adically  bounded under $f_{\l}(z)$ we have that $\hhat_{\lambda,v}(\alpha_i(\l)) = 0$. As $\C_v$ is a dense subspace of $\Aberk{\C_v}$, continuity in $\l$ implies that  the local canonical height function $\hhat_{\lambda,v}(\alpha_i(\l))$ has a natural extension on $\Aberk{\C_v}$.  It follows that $\l\in M_{\alpha_i,v}$ if and only if $\hhat_{\lambda,v}(\alpha_i(\l)) = 0$ (for more details, see \cite{Matt, GHT-ANT}).  Thus, $M_{\alpha_i,v}$ is a closed subset of $\Aberk{\C_v}$ and in fact, one can show that $M_{\alpha_i,v}$ is a compact subset of $\Aberk{\C_v}$. Associated with $\alpha_i$ (for $i=1,2$), we define 
\begin{equation}
\label{eq:G}
G_{\alpha_i,v}(\lambda):= \lim_{m\to\infty} \frac{\log^+|f^{n_i+m}_\lambda(\alpha_i(\l))|_v}{D_{i}\cdot d^m}=\frac{d^{n_i}}{D_{i}}\cdot  \hhat_{\lambda,v}(\alpha_i(\l)).
\end{equation}
Note that $G_{\alpha_i,v}(\l) \ge 0$ for all $\l\in \Aberk{\C_v}$; also,  $\l\in M_{\alpha_i,v}$ if and only if $G_{\alpha_i,v}(\l) = 0$.   It turns out that $G_{\alpha_i,v}$ is the Green's function for $M_{\alpha_i,v}$ relative to $\infty.$ 


\subsection{An important equidistribution theorem for points of small height}

We define (following \cite{Matt, GHT-ANT}) the \emph{generalized ad\`elic Mandelbrot set}  $\M_{\alpha_i} = \prod_{v\in \Omega} M_{\alpha_i,v}$  associated with $\alpha_i$ (for $i=1,2$). Then for each $\lambda \in \Qbar$, we set 
\begin{equation}
\label{eq:height_adelic}
h_{\M_{\alpha_i}}(\l):=h_{\M_{\alpha_i}}(S) = \sum_{v\in\Omega} \left(\frac{1}{|S|}\sum_{\l\in S}G_{\alpha_i,v}(\l)\right) 
\text{ where $S$ is the $\gal(\Qbar/L)$-orbit of $\l$}.
\end{equation}  
As proven in \cite{Matt, GHT-ANT}, we have that $\M_{\alpha_i}$ (for $i=1,2$) is a compact Berkovich ad\`elic set with the \emph{logarithmic capacity} $\g\left(\M_{\alpha_i}\right)=1$ and $h_{\M_{\alpha_i}}(\l)$ represents the height of $\l$ relative to $\M_{\alpha_i}.$ Consequently, the equidistribution result \cite[Theorem~7.52]{BR} applies. The following is a special case we need for our application: 

\begin{theorem}
 \label{thm:equi}
With the above notation, let $\E = \prod_{v\in \Omega} E_v$ be a compact Berkovich ad\`elic set
 with $\g(\E)=1.$ Suppose that 
 $S_n$ is a sequence of $\gal(\Qbar/L)$-invariant finite subsets of
 $\Qbar$ with $|S_n|\to \infty$ and $h_{\E}(S_n) \to 0$ as $n\to
 \infty$.  
For each $v\in\Omega_L$ and for each $n$ let $\d_{v,n}$ be the discrete
 probability measure on $\Aberk{\C_v}$ supported equally on the elements of $S_n$. Then for each $v\in\Omega_L$, the
 sequence of measures $\{\d_{v,n}\}_{n\in\N}$ converges weakly to the
 equilibrium measure $\mu_v$ on $E_v$. 
\end{theorem}

\begin{remark}
We note that similar equidistribution theorems for points of small height were also obtained in various other settings \cite{Bilu, CL, FR, Yuan, Zhang}. These aforementioned results (along with the one from \cite{B-R}) have been the catalyst for a long series of important recent results in arithmetic dynamics. 
\end{remark}


\subsection{Proof of Theorem~\ref{thm:iff}}

Using Theorem~\ref{thm:equi}, we can prove the main result of Section~\ref{subsec:Berkovich}.  Working with the notation and hypotheses from Theorem~\ref{thm:iff} (see also our setup from  Section~\ref{subsec:note}), we first state a useful fact, which was obtained in \cite[Corollary~6.12]{GHT-ANT}.
\begin{lemma}
\label{lem:global_height}
For each $i=1,2$ and for each $\l\in\Lbar$, we have $h_{\M_{\alpha_i}}(\l)= \frac{d^{n_i}}{D_i}\cdot \hhat_{\l}(\alpha_i(\l))$.
\end{lemma}
Hence, using Lemma~\ref{lem:global_height}, equation~\eqref{eq:Bogomolov} yields 
\begin{equation}
\label{eq:4000}
\lim_{n\to\infty}h_{\M_{\alpha_1}}(\l_n)= \lim_{n\to\infty}h_{\M_{\alpha_2}}(\l_n)=0.
\end{equation}
Equation~\eqref{eq:4000} shows that the hypothesis from Theorem~\ref{thm:equi} holds and therefore, we conclude that $M_{\alpha_1,v}=M_{\alpha_2,v}$ for each place $v\in \Omega$. 
Indeed,  for each $n\in\N$, we may
take $S_n$ be the union of the sets of Galois conjugates for $\l_m$
for all $1\le m\le n$. Clearly $|S_n|\to\infty$ as $n\to\infty$, and
also each $S_n$ is $\gal(\Qbar/L)$-invariant. Thus, we obtain
that $\mu_{M_{\alpha_1,v}}=\mu_{M_{\alpha_2,v}}$  for each $v\in \Omega$ and since they are both
supported on $M_{\alpha_1,v}$ (resp. $M_{\alpha_2,v}$), we also get that
$M_{\alpha_1,v}=M_{\alpha_2,v}$. It follows that the two Green's functions 
$G_{\alpha_1,v}$ and $G_{\alpha_2,v}$ for $M_{\alpha_1,v}$ and $M_{\alpha_2,v}$ are the same. 
By the definitions of $G_{\alpha_i,v}$ (for $i=1,2$; see equation~\eqref{eq:G}), for each $v\in \Omega$  we have 
$$
\frac{d^{n_1}}{D_1}\cdot \hhat_{\l,v}(\alpha_1(\l)) =   G_{\alpha_1,v}(\l)
= G_{\alpha_2,v}(\l) = \frac{d^{n_2}}{D_2}\cdot \hhat_{\l,v}(\alpha_2(\l)) \quad \text{for {\em all} $\l \in \C_v$.}
$$
In particular, we get that for each $\l\in\Qbar$, we have that $\hhat_{\l,v}(\alpha_1(\l))=0$ if and only if $\hhat_{\l,v}(\alpha_2(\l))=0$. Then using Lemma~\ref{lem:clearly} (see also equation \eqref{eq:72}, which gives the decomposition of the global canonical height into a sum of local canonical heights), we get that for each $\l\in\Qbar$,  $\alpha_1(\l)$ is preperiodic for $f_{\l}(z)$ if and only if $\alpha_2(\l)$ is preperiodic for $f_\l(z)$.

This concludes our proof of Theorem~\ref{thm:iff}.


\section{Proof of the direct implication in Theorem~\ref{thm:main}}
\label{sec:proof_direct}

We are finally able to prove the direct implication in Theorem~\ref{thm:main}, which we state below.

\begin{proposition}
\label{prop:direct}
Let $d\ge 2$, let $f_t(z)\in\Qbar[t][z]$ be a family of polynomials of degree $d$ of the form:
\begin{equation}
\label{eq:00-2}
f_t(z)=z^d+\sum_{i=0}^{d-2}c_i(t)\cdot z^i\text{, with $c_i(t)\in\Qbar[t]$ for each $i=0,\dots, d-2$.}
\end{equation}  
Let $\alpha_1(t),\alpha_2(t),\beta(t)\in\Qbar[t]$. Assume the following conditions hold:
\begin{enumerate}
\item[(I)] $c_0(t)$ is not identically equal to $0$;  
\item[(II)] not all polynomials $c_0(t),\dots, c_{d-2}(t)$, $\alpha_1(t),\alpha_2(t),\beta(t)$ are constant;
\item[(III)] there exists no $\gamma\in\Qbar^\ast$ such that $\gamma^{-1}f_t\left(\gamma z\right)=\pm C_d(z)$; and
\item[(IV)] for each $i=1,2$, $\alpha_i$ is not persistent preperiodic under the action of $f_t(z)$. 
\end{enumerate}
If there exist infinitely many $\l\in\Qbar$ such that for some $m,n\in\N$ (depending on $\l$) we have 
\begin{equation}
\label{eq:0-2}
f_\l^m(\alpha_1(\l))=f_\l^n(\alpha_2(\l))=\beta(\l) 
\end{equation}
then at least one of the following conditions holds:
\begin{itemize}
\item[(A)] there exists $\ell\in\N$ and there exists $i\in\{1,2\}$ such that 
\begin{equation}
\label{eq:1-2}
f_t^\ell\left(\alpha_i(t)\right)=\beta(t).
\end{equation} 
\item[(B)] there exist $\ell_1,\ell_2\in\N$  such that
\begin{equation}
\label{eq:2-2}
f_t^{\ell_1}\left(\alpha_{1}(t)\right)=f_t^{\ell_2}(\alpha_{2}(t)).
\end{equation}
\item[(C)] there exists $\tilde{f}_t\in\Qbar[t][z]$ and there exist integers $k\ge 2$ and $\ell_1,\ell_2,\ell_3\ge 1$   such that 
\begin{equation}
\label{eq:C1-2}
f_t=\tilde{f}_t^k\text{ and also, we have} 
\end{equation}
\begin{equation}
\label{eq:C2-2}
\tilde{f}_t^{\ell_1}(\alpha_1(t))= \tilde{f}_t^{\ell_2}(\alpha_2(t))= \tilde{f}_t^{\ell_3}(\beta(t)).
\end{equation}
\item[(D)] there exists $\tilde{f}_t\in\Qbar[t][z]$ and there exist integers $k\ge 2$ and $\ell_1,\ell_2\ge 1$   such that 
\begin{equation}
\label{eq:D1-2}
f_t=\tilde{f}_t^k\text{ and also, we have} 
\end{equation}
\begin{equation}
\label{eq:D2-2}
\tilde{f}_t^{\ell_1}(\alpha_1(t))=\tilde{f}_t^{\ell_2}(\alpha_2(t)).
\end{equation}
Furthermore, there exists $N_0\in\N$ satisfying the following two conditions:
\begin{equation}
\label{eq:D4-2}
\gcd(N_0,k)\mid \ell_2-\ell_1\text{, and }
\end{equation} 
\begin{equation}
\label{eq:D3-2}
\tilde{f}_t^{N_0}(\beta(t))=\beta(t).
\end{equation}
\end{itemize}
\end{proposition}

First, we show (see Lemma~\ref{lem:active} from Section~\ref{subsec:active}) that we may assume in Proposition~\ref{prop:direct} that both $\alpha_1$ and $\alpha_2$ are active points for the family $f_t(z)$.


\subsection{A special case of Proposition~\ref{prop:direct}}
\label{subsec:active}

We show next that Proposition~\ref{prop:direct} holds assuming that at least one of the two starting points $\alpha_1$ or $\alpha_2$ is not active.

\begin{lemma}
\label{lem:active}
With the notation and hypotheses from Proposition~\ref{prop:direct}, assume furthermore that $\alpha_1$ is not an active point for $f_t$. Then Proposition~\ref{prop:direct} holds. 
\end{lemma}

\begin{proof}
First, we note that we know (by hypothesis~(IV) from Proposition~\ref{prop:direct}) that $\alpha_1$ is not persistent preperiodic under the action of $f_t$; also, the hypothesis of Lemma~\ref{lem:active} says that $\alpha_1$ is not active for $f_t$.    Combining this information with Proposition~\ref{prop:nonzero_height} and Lemma~\ref{lem:iso}, we obtain the following two facts:
\begin{itemize}
\item[(a)] $f_t(z)$ is a constant family of polynomials, i.e., $c_0(t),\dots, c_{d-2}(t)\in\Qbar[t]$ are constant polynomials;  and
\item[(b)] $\alpha_1(t)\in\Qbar[t]$ is also a constant polynomial.
\end{itemize}

Now, if also $\alpha_2$ is not an active point for $f_t$, it means that also $\alpha_2(t)$ is a constant polynomial. But then the existence of a \emph{single} $\l\in\Qbar$ such that $f_\l^m(\alpha_1(\l))=f_\l^n(\alpha_2(\l))$ (for some suitable $m,n\in\N$) yields that 
\begin{equation}
\label{eq:180}
f_t^m(\alpha_1(t))=f_t^n(\alpha_2(t)),
\end{equation} 
because we would have that both $f_t$ is a constant family of polynomials and also $\alpha_1,\alpha_2$ are constant starting points. But then equation~\eqref{eq:180} yields the desired conclusion~(B) from Proposition~\ref{prop:direct}. Therefore, we may assume from now on, that $\alpha_2$ is an active point for the constant family $f_t(z)$, i.e., that
\begin{itemize}
\item[(c)] $\deg_t\left(\alpha_2(t)\right)\ge 1$.
\end{itemize}

Similarly, if $\beta(t)$ is a constant polynomial, then assuming there exists some $\l\in\Qbar$ along with some $m\in\N$ such that $f_\l^m(\alpha_1(\l))=\beta(\l)$ yields that
\begin{equation}
\label{eq:181}
f_t^m(\alpha_1(t))=\beta(t),
\end{equation} 
because both $f_t(z)$ is a constant family of polynomials and also, $\alpha_1(t),\beta(t)$ are constant polynomials. Equation \eqref{eq:181} yields that conclusion~(A) from Proposition~\ref{prop:direct} holds, as desired. Therefore, from now on, we may assume 
\begin{itemize}
\item[(d)] $\deg_t(\beta(t))\ge 1$.
\end{itemize}

We identify each polynomial $c_i(t)$ with the constant from $\Qbar$ it represents (which we also denote by $c_i$); also, we identify our constant family $f_t(z)$ of polynomials with 
$$f(z):=z^d+\sum_{i=0}^{d-2}c_i z^i.$$
Then for each $\l\in\Qbar$, we have $f_\l(z)=f(z)$. Furthermore, identifying $\alpha_1(t)$ with the constant in $\Qbar$ it represents (which we also denote by $\alpha_1$), then we also note that hypothesis~(IV) from Proposition~\ref{prop:direct} yields
\begin{itemize}
\item[(e)] $\alpha_1$ is not preperiodic for $f(z)$.
\end{itemize}
Hypotheses of Proposition~\ref{prop:direct}, along with the additional conditions~(a)-(e) from Lemma~\ref{lem:active} show that we can apply Proposition~\ref{prop:constant_family} to $(f(z),\alpha_1,\alpha_2(t),\beta(t))$ and therefore, conclude that there exists $n\in\N$ such that
\begin{equation}
\label{eq:182}
f_t^n(\alpha_2(t))=\beta(t).
\end{equation}
Equation \eqref{eq:182} shows that conclusion~(A) from Proposition~\ref{prop:direct} holds.

This concludes our proof of Lemma~\ref{lem:active}.
\end{proof}

Lemma~\ref{lem:active} allows us to assume from now on that
\begin{equation}
\label{eq:184}
\text{both $\alpha_1$ and $\alpha_2$ are active points for the family of polynomials $f_t(z)$.}
\end{equation}


\subsection{Proof of Proposition~\ref{prop:direct}}

We work with the notation and hypotheses from Proposition~\ref{prop:direct}. Furthermore (see \eqref{eq:184} and Lemma~\ref{lem:active}), we may also assume
\begin{equation}
\label{eq:185}
\text{both $\alpha_1$ and $\alpha_2$ are active points for the family $f_t(z)$ of polynomials.}
\end{equation}
We know  there exist infinitely many $\l\in\Qbar$ such that for some positive integers $m$ and $n$, we have
\begin{equation}
\label{eq:10}
f_\l^m(\alpha_1(\l))=f_\l^n(\alpha_2(\l))=\beta(\l).
\end{equation}
If there exists some  $m\in\N$ such that for infinitely many $\l\in\Qbar$, equation~\eqref{eq:10} holds, then we get that 
\begin{equation}
\label{eq:12}
f_t^m(\alpha_1(t))=\beta(t)
\end{equation}
and therefore, conclusion~(A) from Theorem~\ref{thm:main} must hold. A similar conclusion holds assuming there exists some $n\in\N$ such that for infinitely many $\l\in\Qbar$, equation~\eqref{eq:10} holds, then we get that
\begin{equation}
\label{eq:13}
f_t^n(\alpha_2(t))=\beta(t).
\end{equation}
Thus, from now on, we may assume there exists an infinite sequence of distinct parameters $\{\l_j\}_{j\in\N}$ in $\Qbar$ such that for some positive integers $m_j$ and $n_j$ (depending on $\l_j$) satisfying 
\begin{equation}
\label{eq:14}
\lim_{j\to\infty}m_j=\lim_{j\to\infty}n_j=\infty,
\end{equation}
we have that for each $j\ge 1$, the following holds: 
\begin{equation}
\label{eq:15}
f_{\l_j}^{m_j}(\alpha_1(\l_j))=f_{\l_j}^{n_j}(\alpha_2(\l_j))=\beta(\l_j).
\end{equation}

Equations~\eqref{eq:14}~and~\eqref{eq:15} along with \eqref{eq:185}  allow us to apply Proposition~\ref{prop:heights} to deduce that 
\begin{equation}
\label{eq:240}
\lim_{j\to\infty}\hhat_{\l_j}\left(\alpha_1(\l_j)\right)= \lim_{j\to\infty} \hhat_{\l_j}\left(\alpha_2(\l_j)\right) = 0.
\end{equation}
Then equation \eqref{eq:240} (coupled with condition~\eqref{eq:185}) allows us to apply Theorem~\ref{thm:iff}, which yields that for \emph{each} $\l\in\Qbar$, we have that $\alpha_1(\l)$ is preperiodic for $f_\l(z)$ \emph{if and only if} $\alpha_2(\l)$ is preperiodic for $f_\l(z)$. Proposition~\ref{prop:real_converse} applied to $f_t(z)$ with starting point $\alpha_1(t)$ and target point also $\alpha_1(t)$  yields the existence of infinitely many $\l\in\Qbar$ such that $\alpha_1(\l)$ is periodic under the action of $f_\l(z)$, i.e., for some suitable $m\in\N$ (depending on $\l$), we have that $f_\l^m(\alpha_1(\l))=\alpha_1(\l)$. Hence, there exist infinitely many $\l\in\Qbar$ such that both $\alpha_1(\l)$ and $\alpha_2(\l)$ are preperiodic under the action of $f_\l(z)$; alternatively, one can apply \cite[Theorem~1.4]{Laura} to obtain that the set 
\begin{equation}
\label{eq:242}
\left\{\l\colon \text{$\alpha_1(\l)$ and $\alpha_2(\l)$ are preperiodic for $f_\l$}\right\}\text{ is infinite.}
\end{equation}
Property \eqref{eq:242} allows us to apply \cite[Theorem~1.3]{Matt-2}, which yields the existence of a family of polynomials $h_t(z)\in\Qbar[t][z]$ along with some $\ell_0\in\N$ such that 
\begin{equation}
\label{eq:C1-3}
h_{t}\circ f_t^{\ell_0}=f_t^{\ell_0}\circ h_t 
\end{equation}
and furthermore, for some $r,s\in\N$, we have
\begin{equation}
\label{eq:C2-3}
f_t^r(\alpha_1(t))=h_t\left(f_t^s(\alpha_2(t))\right).
\end{equation}
We note that a priori, \cite[Theorem~1.3]{Matt-2} yields that $h_t(z)\in\C[t][z]$, but then using equation \eqref{eq:C1-3} and  Proposition~\ref{prop:poly}, we get that $h_t(z)\in\Qbar[t][z]$. Furthermore, assuming conclusion~(B) from Proposition~\ref{prop:direct} is not met, then we must have some $\tilde{f}_t(z)\in\Qbar[t][z]$ along with some $k\ge 2$ such that 
\begin{equation}
\label{eq:92}
f_t=\tilde{f}_t^k; 
\end{equation}
also, we must have that $h_t=\tilde{f}_t^\ell$ for some $\ell\in\N$. More precisely, equation~\eqref{eq:C2-3} reads now:
\begin{equation}
\label{eq:91}
f_t^r\left(\alpha_1(t)\right) = \tilde{f}_t^\ell\left(f_t^s(\alpha_2(t))\right).
\end{equation}
Clearly, if $k\mid \ell$, then equation~\eqref{eq:92} yields that $h_t$ is an iterate of $f_t$ and thus, condition~(B) would be met. So, we may assume from now on that 
\begin{equation}
\label{eq:assumption_1}
\text{$k$ does not divide $\ell$.} 
\end{equation}
Under assumption \eqref{eq:assumption_1}, we will show that either conclusion~(C) or conclusion~(D) from Proposition~\ref{prop:direct} must hold.

First, we note that there exists an infinite sequence $\{\l_j\}_{j\in\N}$ in $\Qbar$ such that for some suitable positive integers $m_j$ and $n_j$, we have that
\begin{equation}
\label{eq:90}
f_{\l_j}^{m_j}\left(\alpha_1(\l_j)\right)= f_{\l_j}^{n_j}\left(\alpha_2(\l_j)\right) = \beta(\l_j).
\end{equation}
Using equation~\eqref{eq:92}, we restate equation~\eqref{eq:90} as follows:
\begin{equation}
\label{eq:93}
\tilde{f}_{\l_j}^{km_j}\left(\alpha_1(\l_j)\right) = \tilde{f}_{\l_j}^{kn_j}\left(\alpha_2(\l_j)\right) = \beta(\l_j)\text{ for each $j\in\N$.}
\end{equation}
Furthermore, equations~\eqref{eq:91}~and~\eqref{eq:92} yield that we also have:
\begin{equation}
\label{eq:94}
\tilde{f}_{\l_j}^{rk}\left(\alpha_1(\l_j)\right) = \tilde{f}_{\l_j}^{sk+\ell}\left(\alpha_2(\l_j)\right)\text{ for each $j\in\N$.}
\end{equation}
Since $m_j,n_j\to\infty$ as $j\to\infty$ (see \eqref{eq:14}), at the expense of replacing $\{\l_j\}_{j\in\N}$ with an infinite subsequence (simply disregarding finitely many elements from our original sequence $\{\l_j\}$), we may assume that for \emph{each} $j\in\N$, we have
\begin{equation}
\label{eq:95}
m_j,n_j>\max\{r,s,\ell\}.
\end{equation}
Inequality \eqref{eq:95} allows us to apply $\tilde{f}_{\l_j}^{k(m_j-r)}$ to the equation~\eqref{eq:94} and thus get:
\begin{equation}
\label{eq:96}
\tilde{f}_{\l_j}^{km_j}(\alpha_1(\l_j))= \tilde{f}_{\l_j}^{k(s+m_j-r)+\ell}(\alpha_2(\l_j))=\beta(\l_j).
\end{equation}
Combining equations \eqref{eq:96} and \eqref{eq:93} yields that
\begin{equation}
\label{eq:97}
\tilde{f}_{\l_j}^{k(s+m_j-r)+\ell}(\alpha_2(\l_j))= \tilde{f}_{\l_j}^{kn_j}(\alpha_2(\l_j))=\beta(\l_j)\text{ for each $j\in\N$.}
\end{equation}
Our assumption that $k$ does not divide $\ell$ (see \eqref{eq:assumption_1}) yields that 
\begin{equation}
\label{eq:100}
k(s+m_j-r)+\ell\ne kn_j.
\end{equation}
Equations \eqref{eq:100} and \eqref{eq:97} yield that $\beta(\l_j)$ must be a periodic point for $\tilde{f}_{\l_j}(z)$ and therefore, it must also be a periodic point for $f_{\l_j}(z)$ (for each $j\in\N$).  Furthermore, 
$\alpha_2(\l_j)$ must be preperiodic under the action of $\tilde{f}_{\l_j}(z)$ for each $j\in\N$; using \eqref{eq:92}, we get that for each $j\in\N$, $\alpha_2(\l_j)$ is preperiodic under the action of $f_{\l_j}(z)$. 

Next, we split our analysis depending on whether $\beta$ is persistent preperiodic for $f_t$, or not.

\begin{lemma}
\label{lem:C}
With the above notation and hypotheses (including assumption \eqref{eq:assumption_1}) for $f_t,\tilde{f}_t,\alpha_1,\alpha_2,\beta,k,\ell, r, s$, if $\beta$ is not persistent preperiodic under the action of $f_t$, then conclusion~(C) from Proposition~\ref{prop:direct} must hold.
\end{lemma}

\begin{proof}[Proof of Lemma~\ref{lem:C}.]
So, we know that for each $j\in\N$, we have that both $\alpha_2(\l_j)$ and $\beta(\l_j)$ are preperiodic under the action of $f_{\l_j}(z)$.  Then  applying once again \cite[Theorem~1.3]{Matt-2}, this time to $\alpha_2$ and $\beta$ (note our hypothesis from Lemma~\ref{lem:C} that $\beta$ is not persistent preperiodic), we obtain that there exists $g_t(z)\in\Qbar[t][z]$ commuting with an iterate of $f_t(z)$ such that for some $r_2,s_2\in\N$, we have
\begin{equation}
\label{eq:98}
f_t^{r_2}(\alpha_2(t))=g_t\left(f_t^{s_2}(\beta(t))\right).
\end{equation}
Another application of Proposition~\ref{prop:poly} yields that $g_t=\tilde{f}_t^{\ell_2}$ for some $\ell_2\in\N$. Combining equations \eqref{eq:98} and \eqref{eq:91}, along with equation~\eqref{eq:92} delivers the desired conclusion~(C) from Proposition~\ref{prop:direct}. 
\end{proof}

So, we are left with one more case: still under the hypothesis~\eqref{eq:assumption_1}, we assume from now on that $\beta$ is persistent preperiodic. Next, we show that in this case,  conclusion~(D) of Proposition~\ref{prop:direct} must hold.

\begin{lemma}
\label{lem:D}
With the above notation and hypotheses (including assumption~\eqref{eq:assumption_1}) for $f_t,\tilde{f}_t,\alpha_1,\alpha_2,\beta,k,\ell, r, s$, if $\beta$ is persistent preperiodic under the action of $f_t$, then conclusion~(D) from Proposition~\ref{prop:direct} must hold.
\end{lemma}

\begin{proof}[Proof of Lemma~\ref{lem:D}.]
As noted before, equations~\eqref{eq:97}~and~\eqref{eq:100} yield that $\beta(\l_j)$ is a periodic point for $f_{\l_j}(z)$. Since we assume in Lemma~\ref{lem:D} that $\beta$ is persistent preperiodic for $f_t$ (and thus for $\tilde{f}_t$), we claim that it must actually be persistent periodic for $f_t$ (and also for $\tilde{f}_t$), i.e., there exists $N_0\in\N$ such that
\begin{equation}
\label{eq:102}
\tilde{f}_t^{N_0}(\beta(t))=\beta(t).
\end{equation}
Indeed, if $\beta$ is persistent strictly preperiodic then there exist integers $1\le N_1<N_2$ such that 
\begin{equation}
\label{eq:103}
\text{for $0\le i\le N_2-1$, the points $\tilde{f}_t^i(\beta(t))$ are distinct, and $\tilde{f}_t^{N_1}(\beta(t))=\tilde{f}_t^{N_2}(\beta(t))$.}
\end{equation}
So, because  $\beta(\l_j)$ is periodic, then there must be some $i_j\in\{N_1,\dots, N_2-1\}$ such that
\begin{equation}
\label{eq:104}
\beta(\l_j)=\tilde{f}_{\l_j}^{i_j}(\beta(\l_j)).
\end{equation}
Using the pigeonhole principle, there exists some $i_0\in\{N_1,\dots,N_2-1\}$ such that
\begin{equation}
\label{eq:105}
\beta(\l_j)=\tilde{f}_{\l_j}^{i_0}(\beta(\l_j))\text{ for infinitely many $j\in\N$.}
\end{equation}
But then equation~\eqref{eq:105} yields that $\beta(t)=\tilde{f}_t^{i_0}(\beta(t))$, which contradicts \eqref{eq:103}. Therefore, indeed, we must have that $\beta$ is persistent periodic, i.e., equation~\eqref{eq:102} holds for some suitable positive integer $N_0$. Without loss of generality, we may assume $N_0$ is the minimal period of $\beta(t)$ under the action of $\tilde{f}_t(z)$, i.e.,
\begin{equation}
\label{eq:108}
\text{for $j=0,\dots, N_0-1$, the points $\tilde{f}_t^j\left(\beta(t)\right)$ are distinct, while $\beta(t)=\tilde{f}_t^{N_0}(\beta(t))$.}
\end{equation}
We claim next that for all but finitely many $\l\in\Qbar$, we must have that also $\beta(\l)$ is periodic of minimal period $N_0$. Indeed, if $\beta(\l)$ (for some $\l\in\Qbar$) has minimal period less than $N_0$, then  we must have some integer $i_\l\in\{1,\dots, N_0-1\}$, such that
\begin{equation}
\label{eq:106}
\beta(\l)=\tilde{f}_\l^{i_\l}(\beta(\l)).
\end{equation}
So, if equation~\eqref{eq:106} were to hold for infinitely many $\l\in\Qbar$, then we could find some $j_0\in\{1,\dots, N_0-1\}$ such that
\begin{equation}
\label{eq:107}
\beta(\l)=\tilde{f}_\l^{j_0}(\beta(\l))\text{ for infinitely many $\l\in\Qbar$.}
\end{equation}
But then equation~\eqref{eq:107} would force that $\beta(t)=\tilde{f}_t^{j_0}(\beta(t))$, which contradicts \eqref{eq:108}. Therefore, indeed, we must have that for all but finitely many $\l\in\Qbar$, the point $\beta(\l)$ has minimal period $N_0$ under the action of $\tilde{f}_\l(z)$. So, at the expense of replacing the sequence $\{\l_j\}_{j\in\N}$ by an infinite subsequence (simply, by disregarding finitely many elements of the original sequence), we may assume that 
\begin{equation}
\label{eq:109}
\text{for \emph{each} $j\in\N$, the point $\beta(\l_j)$ has minimal period $N_0$ under the action of $\tilde{f}_{\l_j}(z)$.}
\end{equation}
Equations \eqref{eq:109} and \eqref{eq:97} yield that 
\begin{equation}
\label{eq:110}
N_0\mid k\cdot \left(m_j-n_j+s-r\right)+\ell.
\end{equation} 
The divisibility from \eqref{eq:110} yields that $\gcd(N_0,k)\mid \ell$, as desired in conclusion~(D) from Proposition~\ref{prop:direct}.
\end{proof}

Lemma~\ref{lem:D} finishes our proof of Proposition~\ref{prop:direct}. Furthermore, as noted before, Proposition~\ref{prop:direct} coupled with Proposition~\ref{prop:converse} finish the proof for Theorem~\ref{thm:main}.



\begin{thebibliography}{BDNSW25}

\bibitem[AR04]{A-R}
N. Ailon and Z. Rudnick, \emph{Torsion points on curves and common divisors of $a^k-1$ and $b^k-1$}, Acta Arith. \textbf{113} (2004), no.~1, 31--38.

\bibitem[AG26]{A-G}
S. Asgarli and D. Ghioca, \emph{Collision of orbits for families of polynomials defined over fields of positive characteristic}, Math. Z. \textbf{313} (2026), no.~2, Paper No.~29, 39 pp.

\bibitem[Bak09]{Matt-0}
M. Baker, \emph{A finiteness theorem for canonical heights attached to rational maps over function fields}, J. Reine Angew. Math. \textbf{626} (2009), 205--233.

\bibitem[BD11]{Matt}
M.~H.~Baker and L.~DeMarco, \emph{Preperiodic points and unlikely intersections}, Duke Math. J. \textbf{159} (2011), no.~1, 1--29.

\bibitem[BD13]{Matt-2} 
M. H. Baker and L. DeMarco, \emph{Special curves and postcritically finite polynomials}, Forum Math. Pi \textbf{1} (2013), e3, 35~pp.

\bibitem[BR06]{B-R}
M. H. Baker and R. Rumely, \emph{Equidistribution of small points, rational dynamics, and potential theory}, 
Ann. Inst. Fourier (Grenoble) \textbf{56} (2006), no.~3, 625--688.

\bibitem[BR10]{BR}
M. H. Baker and R. Rumely, \emph{Potential theory and dynamics on the Berkovich projective line}, 
Mathematical Surveys and Monographs, \textbf{159}. American Mathematical Society, Providence, RI, 2010. xxxiv+428 pp.

\bibitem[BC16]{Lau-1}
F. Barroero and L. Capuano, \emph{Linear relations in families of powers of elliptic curves}, Algebra Number Theory \textbf{10} (2016), no.~1, 195--214.

\bibitem[BCT24]{Lau-4}
F. Barroero, L. Capuano, and A. Turchet, \emph{Greatest common divisor results on semiabelian varieties and a conjecture of Silverman}, 
Res. Number Theory \textbf{10} (2024), no.~1, Paper No. 17, 16 pp.

\bibitem[BGT16]{BGT-book}
J. P. Bell, D. Ghioca, and T. J. Tucker, \emph{The dynamical Mordell-Lang conjecture}, Mathematical Surveys and Monographs, \textbf{210}. American Mathematical Society, Providence, RI, 2016. xiii+280 pp.

\bibitem[Ben05]{Rob}
R. L. Benedetto, \emph{Heights and preperiodic points of polynomials over function fields}, Int. Math. Res. Not. \textbf{2005}, no.~62, 3855--3866.

\bibitem[BIJMST19]{survey}
R. L. Benedetto, P. Ingram, R. Jones, M. Manes, J.~H.~Silverman, T.~J.~Tucker, \emph{Current trends and open problems in arithmetic dynamics}, Bull. Amer. Math. Soc. (N.S.) \textbf{56} (2019), no.~4, 611--685.

\bibitem[Bil97]{Bilu}
Y. Bilu, \emph{Limit distribution of small points on algebraic tori}, 
Duke Math. J. \textbf{89} (1997), no.~3, 465--476.

\bibitem[BCZ03]{BCZ}
Y. Bugeaud, P. Corvaja and U. Zannier, \emph{An upper bound for the \emph{G.C.D.} of $a^n-1$ and $b^n-1$}, Math. Z. \textbf{243} (2003), 79--84.

\bibitem[CS93]{Call-Silverman}
G. S. Call and J. H. Silverman, \emph{Canonical heights on varieties with morphisms}, Compositio Math. \textbf{89} (1993), no.~2, 163--205.

\bibitem[CL06]{CL}
A.~Chambert-Loir, \emph{Mesures et \'equidistribution sur les espaces de
 {B}erkovich}, J. Reine Angew. Math. \textbf{595} (2006), 215--235.

\bibitem[CZ13]{C-Z-2}
P. Corvaja and U. Zannier, \emph{Greatest common divisors of $u-1$, $v-1$ in positive characteristic and rational points on curves over finite fields}, J. Eur. Math. Soc. (JEMS) \textbf{15} (2013), no.~5, 1927--1942.

\bibitem[DeM16]{Laura}
L. DeMarco, \emph{Bifurcations, intersections, and heights}, 
Algebra Number Theory \textbf{10} (2016), no.~5, 1031--1056.

\bibitem[FG22]{FG}
C. Favre and T. Gauthier, \emph{The arithmetic of polynomial dynamical pairs}, Annals of Mathematics Studies, \textbf{214}, Princeton University Press, Princeton, NJ, 2022, xvii+232 pp. 

\bibitem[FR06]{FR}
C.~Favre and J.~Rivera-Letelier, \emph{\'Equidistribution quantitative des points de petite hauteur sur la droite projective}, Math. Ann. \textbf{355} (2006), no. 2, 311--361.

\bibitem[Ghi14]{G-DML}
D. Ghioca, \emph{The dynamical Mordell-Lang conjecture}, CMS Notes \textbf{46} (2014), no.~3, 14--15.

\bibitem[Ghi17]{G-Survey}
D. Ghioca, \emph{Unlikely intersections in arithmetic dynamics}, 
CMS Notes \textbf{49} (2017), no.~4, 12--13.

\bibitem[Ghi24]{G-JNT}
D.~Ghioca, \emph{Collision of orbits for a one-parameter family of Drinfeld modules}, J. Number Theory \textbf{257} (2024), 320--340.

\bibitem[Ghi26]{1}
D.~Ghioca, \emph{Simultaneously preperiodic points for a family of polynomials in positive characteristic}, Canad. J. Math. \textbf{78} (2026), no.~1, 199--221.

\bibitem[GH13]{AA}
D. Ghioca and L.-C. Hsia, \emph{Torsion points in families of Drinfeld modules}, Acta Arith. \textbf{161} (2013), no.~3, 219--240.

\bibitem[GHT13]{GHT-ANT} 
D. Ghioca, L.-C. Hsia, and T. J. Tucker, \emph{Preperiodic points for families of polynomials}, Algebra Number Theory \textbf{7} (2013), no.~3, 701--732.

\bibitem[GHT17]{GHT-NJM}
D. Ghioca, L.-C. Hsia, and T. J. Tucker, \emph{On a variant of the Ailon-Rudnick theorem in finite characteristic}, 
New York J. Math. \textbf{23} (2017), 213--225.

\bibitem[GHT18]{GHT-PJM}
D. Ghioca, L.-C. Hsia, and T. J. Tucker, \emph{A variant of a theorem by Ailon-Rudnick for elliptic curves}, 
Pacific J. Math. \textbf{295} (2018), no.~1, 1--15.

\bibitem[GM13]{G-M}
D. Ghioca and N. M. Mavraki, \emph{Variation of the canonical height in a family of rational maps},  
New York J. Math. \textbf{19} (2013), 873--907.

\bibitem[GN18]{ANT}
D. Ghioca and K. D. Nguyen, \emph{A dynamical variant of the Pink-Zilber conjecture}, Algebra Number Theory \textbf{12} (2018), no.~7, 1749--1771.


\bibitem[GS]{D-N}
D. Ghioca and N. Shadgar, \emph{Collision of orbits on an elliptic surface}, Bull. Aust. Math. Soc., 2026, 15 pp. (to appear).

\bibitem[GT21]{GT-BAMS}
D. Ghioca and T. J. Tucker, \emph{A reformulation of the dynamical Manin-Mumford conjecture}, Bull. Aust. Math. Soc. \textbf{103} (2021), no.~1, 154--161.

\bibitem[GTZ11]{GTZ}
D. Ghioca, T. J. Tucker, and S. Zhang, \emph{Towards a dynamical Manin-Mumford conjecture}, Int. Math. Res. Not. IMRN \textbf{2011}, no.~22, 5109--5122.

\bibitem[HT17]{HT}
L.-C. Hsia and T. J. Tucker, \emph{Greatest common divisors of iterates of polynomials}, 
Algebra Number Theory \textbf{11} (2017), no.~6, 1437--1459.


\bibitem[Ing13]{Patrick-Crelle}
P. Ingram, \emph{Variation of the canonical height for a family of polynomials}, J. Reine Angew. Math. \textbf{685} (2013), 73--97.

\bibitem[Lau84]{Laurent}
M. Laurent, \emph{\'{E}quations diophantiennes exponentielles}, 
Invent. Math. \textbf{78} (1984), no.~2, 299--327.

\bibitem[LN]{LN}
J.~Lee and G.~Nam, \emph{On simultaneously preperiodic points for one-parameter families of polynomials in characteristic $p$}, preprint available online at https://arxiv.org/pdf/2509.15079.

\bibitem[Luc05]{Luca}
F. Luca, \emph{On the greatest common divisor of $u-1$ and $v-1$ with $u$ and $v$ near $S$-units}, Monatsh. Math. \textbf{146} (2005), no.~3, 239--256. 

\bibitem[MZ10]{M-Z-1}
D. W.~Masser and U. Zannier, \emph{Torsion anomalous points and families of elliptic curves}, Amer. J. Math. \textbf{132} (2010), no.~6, 1677--1691.

\bibitem[MZ12]{M-Z-2}
D.~W.~Masser and U.~Zannier, \emph{Torsion points on families of squares of elliptic curves}, Math. Ann. \textbf{352} (2012), no.~2, 453--484.

\bibitem[MS14]{Alice}
A. Medvedev and T. Scanlon, \emph{Invariant varieties for polynomial dynamical systems}, Ann. of Math. (2) \textbf{179} (2014), no.~1, 81--177.

\bibitem[Sil83]{Sil83}
J. H. Silverman, \emph{Heights and the specialization map for families of abelian varieties}, J. Reine Angew. Math. \textbf{342} (1983), 197--211.

\bibitem[Sil91]{Sil-K3}
J. H. Silverman, \emph{Rational points on K3 surfaces: a new canonical height}, Invent. Math. \textbf{105} (1991), no.~2, 347--373.

\bibitem[Sil04]{Silverman-3}
J. H. Silverman, \emph{Common divisors of elliptic divisibility sequences over function fields}, Manuscripta Math. \textbf{114} (2004), no.~4, 431--446.

\bibitem[Sil07]{Silverman}
J. H. Silverman, \emph{The arithmetic of dynamical systems}, 
Graduate Texts in Mathematics, \textbf{241}. Springer, New York, 2007. x+511 pp.

\bibitem[Tat83]{Tate}
J. Tate, \emph{Variation of the canonical height of a point depending on a parameter}, Amer. J. Math. \textbf{105} (1983), no.~1, 287--294.

\bibitem[Yua08]{Yuan}
X.~Yuan, \emph{Big line bundles over arithmetic varieties}, Invent. Math. \textbf{173} (2008), no.~3, 603--649.

\bibitem[Zan12]{Umberto}
U.~Zannier, \emph{Some problems of unlikely intersections in arithmetic and geometry. With appendixes by David Masser}, Annals of Mathematics Studies, \textbf{181}. Princeton University Press, Princeton, NJ, 2012. xiv+160 pp.

\bibitem[Zha95]{Zhang}
S.-W. Zhang, \emph{Equidistribution of small points on abelian varieties}, Ann. of Math. (2) \textbf{147} (1998), no.~1, 159--165.


\end{thebibliography}
\end{document}